\documentclass[paper]{siamart171218}
\usepackage{url}
\usepackage{amssymb}
\usepackage{amsmath}
\usepackage{stmaryrd}
\usepackage{commath}
\usepackage[numbers,square]{natbib}
\usepackage{graphicx}
\usepackage{tikz}

\makeatletter
\newcommand{\tnorm}{\@ifstar\@tnorms\@tnorm}
\newcommand{\@tnorms}[1]{%
  \left|\mkern-1.5mu\left|\mkern-1.5mu\left|
   #1
  \right|\mkern-1.5mu\right|\mkern-1.5mu\right|
}
\newcommand{\@tnorm}[2][]{%
  \mathopen{#1|\mkern-1.5mu#1|\mkern-1.5mu#1|}
  #2
  \mathclose{#1|\mkern-1.5mu#1|\mkern-1.5mu#1|}
}
\makeatother

\title{Analysis of a space--time hybridizable discontinuous Galerkin
  method for the advection--diffusion problem on time-dependent
  domains \thanks{\funding{Rhebergen gratefully acknowledges support
      from the Natural Sciences and Engineering Research Council of
      Canada through the Discovery Grant program (RGPIN-05606-2015)
      and the Discovery Accelerator Supplement
      (RGPAS-478018-2015). Kirk gratefully acknowledges support from
      the Natural Sciences and Engineering Research Council of Canada
      through the Alexander Graham Bell CGS-M grant. Part of this
      research was conducted while Cesmelioglu was a visiting
      researcher at the Department of Applied Mathematics of the
      University of Waterloo, Canada.}}}
\author{K.L.A. Kirk\thanks{Department of Applied Mathematics,
    University of Waterloo, Waterloo, Ontario N2L~3G1, Canada
    (\email{k4kirk@uwaterloo.ca}, \email{thorvath@uwaterloo.ca}, \email{srheberg@uwaterloo.ca}).}
  \and T.L. Horvath\footnotemark[1]
  \and A. Cesmelioglu\thanks{Department of Mathematics and
    Statistics, Oakland University, Rochester, Michigan 48309, USA (\email{cesmelio@oakland.edu}).}
  \and S. Rhebergen\footnotemark[1]}
\headers{Space--time HDG for advection--diffusion}{K. Kirk,
  T. Horvath, A. Cesmelioglu and S. Rhebergen}
\begin{document}
\maketitle
\begin{abstract}
  This paper presents the first analysis of a space--time hybridizable
  discontinuous Galerkin method for the advection--diffusion problem
  on time-dependent domains. The analysis is based on non-standard
  local trace and inverse inequalities that are anisotropic in the
  spatial and time steps. We prove well-posedness of the discrete
  problem and provide \emph{a priori} error estimates in a
  mesh-dependent norm. Convergence theory is validated by a numerical
  example solving the advection--diffusion problem on a time-dependent
  domain for approximations of various polynomial degree.
\end{abstract}
\section{Introduction}
\label{sec:introduction}

Many important physical processes are governed by the solution of
time-dependent partial differential equations on moving and deforming
domains.  Of particular importance are advection dominated transport
problems, with applications ranging from multi-phase flows separated
by evolving interfaces to incompressible flow problems arising from
fluid-structure interaction. The presence of dynamic meshes introduces
additional challenges in the design of numerical methods. Most notable
of these challenges is the Geometric Conservation Law
(GCL)~\cite{Lesoinne:1996}, which requires that uniform flow solutions
remain uniform under grid motion. Satisfaction of the GCL is not
trivial, as observed for the popular Arbitrary Lagrangian-Eulerian
(ALE) class of methods in which the time-varying domain is mapped to a
fixed reference domain. By performing all computations on the
reference domain, the ALE method allows the use of explicit or
multi-step time-stepping schemes. However, additional constraints must
be placed on the algorithm to satisfy the GCL that often are not met
in practice for arbitrary mesh movements, see
e.g.,~\cite{Persson:2009}.

In contrast, the space--time discontinuous Galerkin (DG) method
inherently satisfies the GCL, as shown in~\cite{Lesoinne:1996}. Rather
than mapping to a fixed reference domain, the problem is recast into a
space--time domain in which spatial and temporal variables are not
distinguished. This space--time domain is then partitioned into slabs
and each slab is discretized using discontinuous basis functions in
both space and time. The result is a fully conservative scheme that
automatically accounts for grid movement and can be made arbitrarily
higher-order accurate in space and time. The space--time DG method is
suitable for both hyperbolic and parabolic problems, and has been
successfully applied to diffusion and advection--diffusion
problems~\cite{Cangiani:2017,Feistauer:2007,Feistauer:2011,Sudirham:2006},
as well as compressible and incompressible flow
problems~\cite{Rhebergen:2013b, Sudirham:2008,Vegt:2002}.

We pause to mention two possible solution procedures for space--time
methods. The first, which we do not pursue in this article, constructs
a single global system for the solution in the whole space--time
domain~\cite{Neumueller:2013}. The advantage of solving the PDE in the
space--time domain all-at-once is the applicability of
parallel-in-time methods~\cite{McDonald:2018}. The second, outlined
in~\cite{Jamet:1978, Masud:1997, Vegt:2002}, instead computes the
solution slab-by-slab. We apply this second approach.

In the space--time DG method, a time-dependent $d$-dimensional partial
differential equation is discretized by DG in $d+1$ dimensions
resulting in substantially more degrees-of-freedom in each element
compared to traditional time-stepping approaches. To alleviate the
computational burden of the space--time DG method, Rhebergen and
Cockburn~\cite{Rhebergen:2012,Rhebergen:2013} introduced the
space--time Hybridizable DG (HDG) method, extending the HDG method
of~\cite{Cockburn:2009} to space--time.

The HDG method reduces the number of globally coupled
degrees-of-freedom by first introducing approximate traces of the
solution on element facets. Enforcement of continuity of the normal
component of the numerical flux across element facets allows for the
unique determination of these approximate traces. The resulting linear
system may then be reduced through static condensation to a global
system of algebraic equations for only these approximate traces. In
essence, the dimension of the problem is reduced by one, since the
number of globally coupled degrees of freedom is of order
$\mathcal{O}(p^d)$ instead of $\mathcal{O}(p^{d+1})$, where $p$
denotes the polynomial order and $d$ is the spatial dimension of the
problem under consideration.

The first space--time HDG methods for partial differential equations
on time-dependent domains were introduced
in~\cite{Rhebergen:2012,Rhebergen:2013}, however, to the authors'
knowledge, these methods have not been analyzed. The analysis on fixed
domains of the space--time DG method, however, has been considered for
both linear and nonlinear advection--diffusion problems
in~\cite{Feistauer:2007,Feistauer:2011}, and for the space--time HDG
method in~\cite{Neumueller:2013}. Recently, for a time-dependent
diffusion equation,~\cite{Cangiani:2017} analyzed the heat equation on
prismatic space--time elements.

The consideration of moving domains significantly alters the analysis
of the method compared to fixed domains. In particular, moving meshes
lack the tensor product structure necessary to use the space--time
projection introduced in~\cite{Feistauer:2007,Feistauer:2011}, or the
inverse and trace inequalities derived in~\cite{Cangiani:2017} without
modification. The first error analysis of a space--time DG method on
moving and deforming domains for the linear advection--diffusion
equation was performed in~\cite{Sudirham:2005}, and for the Oseen
equations in~\cite{Sudirham:2008}, laying the groundwork for the error
estimates in~\cref{s:erroranalysis}.

In this paper, we analyze a space--time HDG method for the
advection--diffusion equation on a time-dependent
domain. In~\cref{s:advdif} we discuss the scalar advection--diffusion
problem in a space--time setting. Next, in~\cref{s:sthdf}, we discuss
the finite element spaces necessary to obtain the weak formulation of
the advection--diffusion problem, which we subsequently
introduce.~\Cref{s:consstabbound} deals with the consistency and
stability of the space--time HDG method.  Theoretical rates of
convergence of the space--time HDG formulation in a mesh-dependent
norm on moving grids are derived
in~\cref{s:erroranalysis}. Finally,~\cref{s:numex} presents the
results of a numerical example to support the theoretical analysis,
and a concluding discussion is given in~\cref{s:conclusions}.

\section{The advection--diffusion problem}
\label{s:advdif}

Let $\Omega(t) \subset \mathbb{R}^d$ be a time-dependent polygonal
($d=2$) or polyhedral ($d=3$) domain whose evolution depends
continuously on time $t \in [t_0, t_N]$. Let $x=\del{x_1,\cdots,x_d}$
be the spatial variables and denote the spatial gradient operator by
$\overline{\nabla} = \del{\partial_{x_1},\cdots,\partial_{x_d}}$. We
consider the time-dependent advection--diffusion problem
\begin{equation}
  \label{eq:advdif_tdepdomain}
  \partial_t u + \overline{\nabla} \cdot(\bar{\beta} u) - \nu \overline{\nabla}^2u = f 
  \qquad \mbox{in}\ \Omega(t),\ t_0 < t \le t_N,
\end{equation}
with given advective velocity $\bar{\beta}$, forcing term $f$ and
constant and positive diffusion coefficient $\nu$.

Before introducing the space--time HDG method in \cref{s:sthdf}, we
first present the space--time formulation of the advection--diffusion
problem \cref{eq:advdif_tdepdomain}. Let
$ \mathcal{E} := \{(t,x) : x \in \Omega(t), \; t_0 < t < t_N \}
\subset \mathbb{R}^{d+1}$
be a $(d+1)$-dimensional polyhedral space--time domain.  We denote the
boundary of $\mathcal{E}$ by $\partial \mathcal{E}$, and note that it
is comprised of the hyper-surfaces
$\Omega(t_0) := \{(t,x) \in \partial \mathcal{E} : \; t=t_0 \}$,
$\Omega(t_N) := \{(t,x) \in \partial \mathcal{E} : \; t=t_N \}$, and
$\mathcal{Q}_{\mathcal{E}} := \{(t,x) \in \partial \mathcal{E}: \; t_0
< t < t_N\}$.
The outward space--time normal vector to $\partial \mathcal{E}$ is
denoted by $n := (n_t, \bar{n})$, where $n_t$ and $\bar{n}$ are the
temporal and spatial parts of the space--time normal vector,
respectively.

To recast the advection--diffusion problem in the space--time setting
we introduce the space--time velocity field $\beta := (1,\bar{\beta})$
and the operator $\nabla := (\partial_t, \overline{\nabla})$. The
space--time formulation of~\cref{eq:advdif_tdepdomain} is then given
by
\begin{equation}
  \label{eq:modprob}
  \nabla \cdot (\beta u) - \nu \overline{\nabla}^2 u = f \quad \mbox{in}\ \mathcal{E},   
\end{equation}
where $f \in L^2(\mathcal{E})$ and where
$\beta, \nabla\cdot\beta \in L^{\infty}(\mathcal{E})$.

We partition the boundary of $\Omega(t)$ such that
$\partial\Omega(t) = \Gamma_D(t) \cup \Gamma_N(t)$ and
$\Gamma_D(t) \cap \Gamma_N(t) = \emptyset$ and we partition the
space--time boundary into
$\partial \mathcal{E} = \partial \mathcal{E}_D \cup \partial
\mathcal{E}_N$,
where
$\partial\mathcal{E}_D := \cbr{(t,x): x \in \Gamma_D(t), t_0 < t \le
  t_N}$
and
$\partial\mathcal{E}_N := \cbr{(t,x): x \in
  \Gamma_N(t)\cup\Omega(t_0), t_0 < t \le t_N}$.
Given a suitably smooth function
$g : \partial\mathcal{E}_N \to \mathbb{R}$, we prescribe the initial
and boundary conditions
\begin{subequations}
  \begin{align}  
    \label{eq:modprob_mixed}
    -\zeta u \beta\cdot n + \nu \overline{\nabla} u \cdot \bar{n} &= g 
    \quad \mbox{on } \partial\mathcal{E}_N, 
    \\ 
    \label{eq:modprob_dirichlet}
    u &= 0 \quad \mbox{on}\ \partial\mathcal{E}_D, 
  \end{align}  
  \label{eq:boundaryinitialconditions}
\end{subequations}
where $\zeta$ is an indicator function for the inflow boundary of
$\mathcal{E}$, i.e., the portions of the boundary where
$\beta \cdot n < 0$. Note that~\cref{eq:modprob_mixed} imposes the
initial condition $u(x,0) = g(x)$ on $\Omega(t_0)$.

\section{The space--time hybridizable discontinuous Galerkin method}
\label{s:sthdf}
In this section we introduce the space--time mesh, the space--time
approximation spaces and the space--time HDG formulation for the
advection--diffusion
problem~\cref{eq:modprob}--\cref{eq:boundaryinitialconditions}.

\subsection{Description of space--time slabs, faces and elements}
\label{ss:stdesc}
We begin this section with a description of the discretization of the
space--time domain.  First, the time interval $[t_{0},t_N]$ is
partitioned into the time levels $t_{0} < t_{1} < \dots < t_N$, where
the $n$-th time interval is defined as $I_{n} = \del{t_{n}, t_{n+1}}$
with length $\Delta t_{n} = t_{n+1} - t_{n}$. For simplicity we will
assume a fixed time interval length, i.e., $\Delta t_n = \Delta t$ for
$n=0,1,\cdots, N-1$. For ease of notation, we will denote
$\Omega(t_n) = \Omega_n$ in the sequel. The space--time domain is then
divided into space--time slabs
$\mathcal{E}^{n} = \mathcal{E} \cap \del{I_{n} \times
  \mathbb{R}^{d}}$.
Each space--time slab $\mathcal{E}^{n}$ is bounded by $\Omega_n$,
$\Omega_{n+1}$, and
$\mathcal{Q}^{n}_{\mathcal{E}} = \partial \mathcal{E}^{n} \setminus
\del{\Omega_n \cup \Omega_{n+1} }$.

We further divide each space--time slab into space--time elements,
$\mathcal{E}^n = \bigcup_{j} \mathcal{K}_{j}^{n}$. To construct the
space--time element $\mathcal{K}_{j}^n$, we divide the domain
$\Omega_n$ into non-overlapping spatial elements $K_{j}^{n}$, so that
$\Omega_n = \bigcup_{j} K_{j}^{n}$. Then, at $t_{n+1}$ the spatial
elements $K_{j}^{n+1}$ are obtained by mapping the nodes of the
elements $K_{j}^{n}$ into their new position via the transformation
describing the deformation of the domain. Each space--time element
$\mathcal{K}_{j}^{n}$ is obtained by connecting the elements
$K_{j}^{n}$ and $K_{j}^{n+1}$ via linear interpolation in time. 

The boundary of the space--time element $\mathcal{K}_j^n$ consists of
$K_{j}^{n}$, $K_{j}^{n+1}$, and
$\mathcal{Q}_{j}^{n} = \partial \mathcal{K}_{j}^{n} \setminus
(K_{j}^{n} \cup K_{j}^{n+1})$.
On $\partial \mathcal{K}_j^n$, the outward unit space--time normal
vector is denoted by
$n^{\mathcal{K}_j^n} = (n_{t}^{\mathcal{K}_j^n},
\bar{n}^{\mathcal{K}_j^n})$,
where $n_t^{\mathcal{K}_j^n}$ and $\bar{n}^{\mathcal{K}_j^n}$ are,
respectively, the temporal and spatial parts of the space--time normal
vector. On $K^{n}_j$, $n^{\mathcal{K}_j^n} = (-1,0)$, while on
$K^{n+1}_j$, $n^{\mathcal{K}_j^n}=(1,0)$. In the remainder of the
article, we will drop the subscripts and superscripts when referring
to space--time elements, their boundaries and outward normal vectors
wherever no confusion will occur.

We complete the description of the space--time domain with the
tessellation $\mathcal{T}_h^n$ consisting of all space--time elements
in $\mathcal{E}^n$, and
$\mathcal{T}_{h} = \bigcup_{n} \mathcal{T}_{h}^{n}$ consisting of all
space--time elements in $\mathcal{E}$. An illustration of a
space--time domain is shown in the case of one spatial dimension in
\cref{fig:ST_dom}.

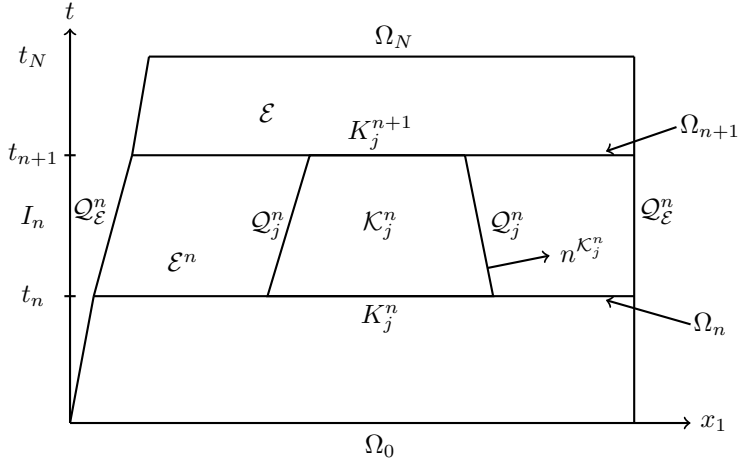
\begin{figure}
  \label{fig:ST_dom}
  \centering
  \begin{tikzpicture}[scale=0.75]
    %
    \draw[->, thick] (4,-2.5)--(4,4.5) node[above]{$t$};
    \draw[->, thick] (4,-2.5)--(15,-2.5) node[right]{$x_1$};
    \draw[thick] (3.9,-0.25)--(4.1,-0.25);
    \draw(3.35,-0.25) node{$t_n$};
    \draw[thick] (3.9,2.25)--(4.1,2.25);
    \draw(3.35,2.25) node{$t_{n+1}$};
    \draw(3.35,1.15) node{$I_n$};
    %
    \draw[thick](11,2.25)--(8.25,2.25)--(7.5,-0.25)--(11.5,-0.25)--(11,2.25);
    \draw(9.5,1) node{$\mathcal{K}_j^n$};
    %
    \draw[->,thick](11.4,0.25)--(12.5,0.47); 
    \draw(13.1, 0.58)node{$n^{\mathcal{K}_j^n}$};
    \draw(6,0.35) node{$\mathcal{E}^n$};
    \draw(7.5,1) node{$\mathcal{Q}_j^n$};
    \draw(11.75,1) node{$\mathcal{Q}_j^n$};
    \draw(9.5,2.65) node{$K^{n+1}_j$};
    \draw(9.5,-0.65) node{$K^{n}_j$};
    \draw[thick] (14,-2.5)--(14,4);
    \draw[thick] (5.4,4)--(14,4);
    \draw[thick] (5.4,4)--(5.1,2.25);
    \draw[thick] (5.1,2.25) -- (4.42,-0.25);
    \draw[thick]  (4.42,-0.25) -- (4,-2.5);
    \draw(3.35,4) node{$t_N$};
    \draw(9.75,4.35) node{$\Omega_N$};
    \draw(4.35,1.25) node{$\mathcal{Q}_{\mathcal{E}}^{n}$};
    \draw(14.4,1.25) node{$\mathcal{Q}_{\mathcal{E}}^{n}$};
    \draw(7.5,3) node{$\mathcal{E}$};
    \draw[thick](5.1,2.25)--(14,2.25);
    \draw[thick](4.42,-0.25)--(14,-0.25);
    \draw[->,thick](15,-0.75)--(13.5,-0.32);
    \draw(15.35,-0.75) node{$\Omega_n$};
    \draw[->,thick](14.75,2.75)--(13.5,2.32);
    \draw(15.35,2.75) node{$\Omega_{n+1}$};
    \draw(9.5,-2.9) node{$\Omega_{0}$};
  \end{tikzpicture}
  \caption{An example of a space--time slab in a polyhedral
    $(1+1)$-dimensional space--time domain.}
\end{figure}

Finally, an interior space--time facet $\mathcal{S}$ is shared by two
adjacent elements $\mathcal{K}^{L}$ and $\mathcal{K}^{R}$,
$\mathcal{S}= \partial \mathcal{K}^{L} \cap \partial \mathcal{K}^{R}$,
while a boundary facet is a face of $\partial \mathcal{K}$ that lies
on $\partial \mathcal{E}$. The set of all facets will be denoted by
$\mathcal{F}$, and the union of all facets by $\Gamma$.

\subsection{Approximation spaces}
\label{ss:approxspaces}
We define the Sobolev space
$H^s(\Omega) = \{ v \in L^2(\Omega): D^{\alpha}v \in L^2(\Omega)
\text{ for } |\alpha| \leq s \}$,
where $D^{\alpha}v$ denotes the weak derivative of $v$, $\alpha$ is
the multi-index symbol, $s$ a non-negative integer, and
$\Omega \subset \mathbb{R}^n$ is an open domain with $n=d$ or $n=d+1$
(see e.g.~\cite{Brenner:book}). The space $H^s(\Omega)$ is equipped
with the following norm and semi-norm:
\begin{equation}
  \norm{v}_{H^s(\Omega)}^2 = 
  \sum_{| \alpha | \leq s} \norm{D^{\alpha} v}_{L^2(\Omega)}^2 
  \quad\text{and}\quad
  | v |_{H^s(\Omega)}^2 = \sum_{| \alpha | = s} \norm{D^{\alpha} v}_{L^2(\Omega)}^2,
\end{equation}
where $\norm{\cdot}_{L^2(\Omega)}$ is the standard $L^2$-norm on
$\Omega$. In the sequel, we will simply write
$\norm{v}_{\Omega} = \norm{v}_{L^2(\Omega)}$.

Next, we introduce anisotropic Sobolev spaces on an open domain
$\Omega \subset \mathbb{R}^{d+1}$~\cite{Georgoulis:2003}. For
simplicity, we follow~\cite{Sudirham:2006, Sudirham:2008} by
restricting the anisotropy to the case where the Sobolev index can
differ only between spatial and temporal variables. All spatial
variables will have the same index. Let $(s_t, s_s)$ be a pair of
non-negative integers, with $s_t$, $s_s$ corresponding to the spatial
and temporal Sobolev indices. For $\alpha_{t}, \alpha_{s_i} \geq 0$,
$i=1,\dots,d$, we define the anisotropic Sobolev space of order
$(s_t, s_s)$ on $\Omega \subset \mathbb{R}^{d+1}$ by
\begin{equation}
  H^{(s_t,s_s)}(\Omega) = \{ v \in L^2(\Omega): D^{\alpha_t}D^{\alpha_s}v \in L^2(\Omega) 
  \text{ for } \alpha_{t} \leq s_t,  |\alpha_s| \leq s_s \},
\end{equation}
where $\alpha_s = \del{\alpha_{s_1},\cdots,\alpha_{s_d}}$. The
anisotropic Sobolev norm and semi-norm are given by, respectively,
\begin{equation*}
  \norm{v}_{H^{(s_t,s_s)} (\Omega)}^2 = \sum_{\substack{\alpha_t \leq s_t \\ | \alpha_s | \leq s_s}} 
  \norm{D^{\alpha_t}D^{\alpha_s} v}_{\Omega}^2 
  \quad\text{and}\quad
  |v|_{H^{(s_t,s_s)} (\Omega)}^2 = \sum_{\substack{\alpha_t = s_t \\ | \alpha_s | = s_s}} 
  \norm{D^{\alpha_t}D^{\alpha_s} v}_{\Omega}^2.  
\end{equation*}

We assume that each space--time element $\mathcal{K}$ is the image of
a fixed master element $\widehat{\mathcal{K}} = \del{-1,1}^{d+1}$
under two mappings. First, we construct an intermediate tensor-product
element $\widetilde{\mathcal{K}}$ from an affine mapping
$F_{\mathcal{K}} : \widehat{\mathcal{K}} \rightarrow
\widetilde{\mathcal{K}}$
of the form $F_{\mathcal{K}}(\hat{x}) = A_{\mathcal{K}}\hat{x} + b$,
where
$A_{\mathcal{K}} = \text{diag}\del{\frac{\Delta t}{2}, \frac{h_1}{2},
  \dots, \frac{h_d}{2}}$.
Here $h_i$ is the edge length in the $i$-th coordinate direction,
$\Delta t$ the time-step, and $b \in \mathbb{R}^{d+1}$ is a constant
vector.

Next, the space--time element $\mathcal{K}$ is obtained from
$\widetilde{\mathcal{K}}$ via the suitably regular diffeomorphism
$\phi_{\mathcal{K}}:\widetilde{\mathcal{K}}\to\mathcal{K}$. The
mapping $\phi_{\mathcal{K}}$ determines the shape of the space--time
element after the size of the element has been specified by
$F_{\mathcal{K}}$. Following \cite{Georgoulis:2003}, we will assume
that the Jacobian of the diffeomorphism $\phi_{\mathcal{K}}$
satisfies:
\begin{equation*}
  C_1^{-1} \leq |\det J_{\phi_{\mathcal{K}}}| \leq C_1, 
  \qquad 
  \norm{\det J_{\phi_{\mathcal{K}}\setminus mn}}_{L^{\infty}(\widetilde{\mathcal{K}})} \leq C_2, 
  \quad m,n = 0,\dots,d,  \quad \forall \mathcal{K} \in \mathcal{T}_h,
\end{equation*}
where $C_1$ and $C_2$ are constants independent of the edge lengths
$h_i$ and the time-step $\Delta t$, and where
$\det J_{\phi_{\mathcal{K}}\setminus mn}$ denotes the $(m,n)$ minor of
$J_{\phi_{\mathcal{K}}}$.
\begin{figure}
  \label{fig:ST_elem}
  \centering
  \begin{tikzpicture}[scale=0.75]
    \draw[->, thick] (-1,4,2)--(-1,5.5,2) node[above]{$x_{2}$};
    \draw[->, thick] (-1,4,2)--(-1,4,0) node[right]{$t$};
    \draw[->, thick] (-1,4,2)--(0.5,4,2) node[right]{$x_{1}$};
    \draw[thick](2,2,0)--(0,2,0)--(0,2,2)--(2,2,2)--(2,2,0)--(2,0,0)--(2,0,2)--(0,0,2)--(0,2,2);
    \draw[thick](2,2,2)--(2,0,2);
    \draw[dotted](2,0,0)--(0,0,0)--(0,2,0);
    \draw[dotted](0,0,0)--(0,0,2);
    \draw(1.125,1.25,2) node{$\widehat{\mathcal{K}}$};
    \draw(-1,-1.7,-1) node{$(-1,-1,-1)$};
    \draw(3.1,3.6,3) node{$(1,1,1)$};
    \draw[->,thick](3.7,2,3) to [out=45] node[above](4.5,2,3){$F_{\mathcal{K}}$} (5.0,2,3);
    \draw[thick](6,0,3)--(8,0,3)--(8,0,0)--(8,4,0)--(6,4,0)--(6,4,3)--(8,4,3)--(8,4,0);
    \draw[thick](8,4,3)--(8,0,3);
    \draw[thick](6,4,3)--(6,0,3);
    \draw[dotted](6,4,0)--(6,0,0)--(8,0,0);
    \draw[dotted](6,0,0)--(6,0,3);
    \draw(7.25,2.25,3) node{$\widetilde{\mathcal{K}}$};
    \draw[<->,thick](6,-0.3,3)--(8,-0.3,3);
    \draw(7,-0.6,3) node{$h_1$};
    \draw[<->,thick](5.7,0,3)--(5.7,4,3);
    \draw(5.3,0.8,3) node{$h_2$};
    \draw[<->,thick](5.7,4,3)--(5.7,4,0);
    \draw(6,5,3) node{$\Delta t$};
    \draw[->,thick](9.5,2.5,3) to [out=45] node[above](9.5,2.5,3){$\phi_{\mathcal{K}}$} (11.5,2.5,3);
    \draw[thick](12,0,3)--(14,0,3)--(14,0,0)--(14.5,3.5,0)--(12,4,0)--(12,4,3)--(14.5,3.5,3)--(14.5,3.5,0);
    \draw[thick](14.5,3.5,3)--(14,0,3);
    \draw[thick](12,4,3)--(12,0,3);
    \draw[dotted](12,4,0)--(12,0,0)--(14,0,0);
    \draw[dotted](12,0,0)--(12,0,3);
    \draw(13.25,2.25,3) node{$\mathcal{K}$};
  \end{tikzpicture}
  \caption{Construction of the space--time element $\mathcal{K}$
    through an affine mapping
    $F_{\mathcal{K}}:\widehat{\mathcal{K}}\to\widetilde{\mathcal{K}}$
    and a diffeomorphism
    $\phi_{\mathcal{K}}:\widetilde{\mathcal{K}}\to\mathcal{K}$~\cite{Sudirham:2006}.}
\end{figure}
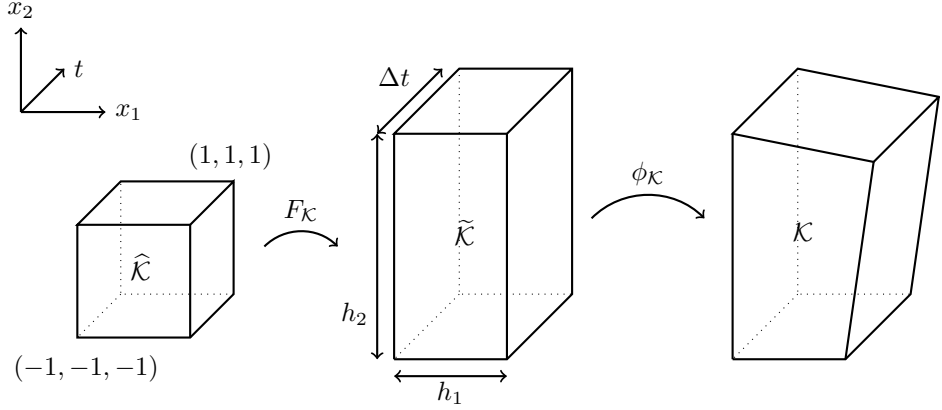

Following~\cite{Sudirham:2006}, we define the Sobolev space
$H^{(s_t,s_s)}(\widetilde{\mathcal{K}})$ as
\begin{equation}
  H^{(s_t,s_s)}(\widetilde{\mathcal{K}}) = \{ v \in L^2(\widetilde{K}): D^{\alpha_t}D^{\alpha_s}v 
  \in L^2(\widetilde{K}) \text{ for } \alpha_{t} \leq s_t,  |\alpha_s| \leq s_s \}.
\end{equation}
Furthermore, the Sobolev space $H^{(s_t,s_s)}(\mathcal{K})$ is defined as
\begin{equation}
  H^{(s_t,s_s)}(\mathcal{K}) = \{ v \in L^2(\mathcal{K}): v \circ \phi_{\mathcal{K}} 
  \in H^{(s_t,s_s)}(\widetilde{\mathcal{K}})\},
\end{equation}
see~\cite[Definition 2.9]{Georgoulis:2003}.

For the analysis in~\cref{s:consstabbound} we require the concept of a
broken anisotropic Sobolev space. We assign to $\mathcal{T}_h$ the
broken Sobolev space
\begin{equation}
  H^{(s_t,s_s)} (\mathcal{T}_h) = \{ v \in L^2(\mathcal{E}): v |_{\mathcal{K}} 
  \in H^{(s_s,s_t)}(\mathcal{K}),
  \forall \mathcal{K} \in \mathcal{T}_h \},
\end{equation}
which we equip with the broken anisotropic Sobolev norm and semi-norm,
respectively,
\begin{equation}
  \norm{v}_{H^{(s_t,s_s)}( \mathcal{T}_h)}^2 = \sum_{\mathcal{K} \in \mathcal{T}_h} 
  \norm{v}_{H^{(s_t,s_s)}(\mathcal{K})}^2
  \quad\text{and}\quad
  |v|_{H^{(s_t,s_s)}(\mathcal{T}_h)}^2 = \sum_{\mathcal{K} \in \mathcal{T}_h} 
  |v|_{H^{(s_t,s_s)}(\mathcal{K})}^2.
\end{equation}
For $v \in H^{(1,1)}(\mathcal{T}_h)$, we define the broken
(space--time) gradient $\nabla_h v$ by
$(\nabla_h v)|_{\mathcal{K}} = \nabla (v|_{\mathcal{K}})$,
$\forall \mathcal{K} \in \mathcal{T}_h$.

Additionally, we will make use of the following (spatial) shape
regularity assumption. Suppose $\mathcal{K} \in \mathcal{T}_h$ is
constructed from the fixed reference element $\widehat{\mathcal{K}}$
via the mappings
$F_{\mathcal{K}}: \widehat{K} \to \widetilde{\mathcal{K}}$ and
$\phi_{\mathcal{K}}: \widetilde{\mathcal{K}} \to \mathcal{K}$. Let
$h_K$ and $\rho_K$ denote the radii of the $d$-dimensional
circumsphere and inscribed sphere of the brick
$h_1 \times \cdots \times h_d$, respectively. We assume the existence
of a constant $c_r > 0$ such that
\begin{equation} 
  \label{eq:shape_reg}
  \frac{h_K}{\rho_K} \le c_r, \quad \forall \mathcal{K} \in \mathcal{T}_h.
\end{equation}

For the HDG method, we require the finite element spaces
\begin{align} 
  V_{h}^{(p_{t},p_{s})} &= \cbr{ v_h \in L^{2}\del{ \mathcal{E} } :
    v_h |_{\mathcal{K}} \circ \phi_{\mathcal{K}} \circ F_{\mathcal{K}}
    \in Q_{(p_{t},p_{s})}(\widehat{\mathcal{K}}), \forall \mathcal{K} \in \mathcal{T}_{h} }, 
  \\
  M_{h}^{(p_{t},p_{s})} &= \{ \mu_h \in L^{2}\del{ \Gamma } : 
  \mu_h |_{\mathcal{S}}\circ \phi_{\mathcal{K}} \circ F_{\mathcal{K}}
  \in Q_{(p_{t},p_{s})}(\widehat{\mathcal{S}}) , \forall  \mathcal{S} \in \mathcal{F}, 
  \\ 
  \notag & \hspace{20em} \mu_h = 0 \text{ on } \partial \mathcal{E}_D \},
\end{align}
where $Q_{(p_t,p_s)}(D)$ denotes the set of all tensor-product
polynomials of degree $p_t$ in the temporal direction and $p_s$ in
each spatial direction on a domain $D$. Furthermore, we define
$V_h^{\star} = V_{h}^{(p_{t},p_{s})} \times M_{h}^{(p_{t},p_{s})}$.

\subsection{Weak formulation}
\label{ss:weakformulation}
It will be convenient to introduce the bilinear forms
\begin{subequations}
  \begin{align}
    \label{eq:bilin_aa}
    a_h^a\del{ (u, \lambda), (v, \mu) } &=
    -\sum_{\mathcal{K}\in\mathcal{T}_h}\int_{\mathcal{K}}\beta u \cdot \nabla_h v\dif x 
    + \int_{\partial\mathcal{E}_N}\frac{1}{2}\del{\beta \cdot n  
    + |\beta \cdot n |}\lambda\mu \dif s
    \\ \notag &
    + \sum_{\mathcal{K}\in\mathcal{T}_h}\int_{\partial\mathcal{K}} \frac{1}{2}\del{
    \beta \cdot n  (u+\lambda)
    + |\beta \cdot n |(u - \lambda)} \del{v - \mu} \dif s,
    \\
    \label{eq:bilin_ad}
    a_h^d\del{ (u, \lambda), (v, \mu) } &=
    \sum_{\mathcal{K}\in\mathcal{T}_h}\int_{\mathcal{K}}\nu\overline{\nabla}_h u \cdot
    \overline{\nabla}_h v \dif x 
    + \sum_{\mathcal{K}\in\mathcal{T}_h}\int_{\mathcal{Q}}\frac{\nu \alpha}{h_K}
    (u - \lambda)\del{v - \mu} \dif s
    \\ \notag &
    - \sum_{\mathcal{K}\in\mathcal{T}_h}\int_{\mathcal{Q}} \sbr{ \nu(u - \lambda)\overline{\nabla}_h 
    v \cdot \bar{n} + \nu\overline{\nabla}_h u\cdot \bar{n} \del{v - \mu}}\dif s,
  \end{align}
\end{subequations}
where $\alpha>0$ is a penalty parameter. The space--time HDG method
for~\cref{eq:modprob}--\cref{eq:boundaryinitialconditions} is then
given by: find $(u_h, \lambda_h) \in V_h^{\star}$ such that
\begin{equation}
  \label{eq:compactwf}
  a_h\del{ (u_h, \lambda_h), (v_h, \mu_h) } = 
  \sum_{\mathcal{K}\in\mathcal{T}_h}\int_{\mathcal{K}}fv_h \dif x + \int_{\partial\mathcal{E}_N} g \mu_h \dif s
  \qquad \forall (v_h, \mu_h) \in V_h^{\star},
\end{equation}
where $  a_h\del{(u, \lambda), (v, \mu)} = a_h^a\del{ (u, \lambda), (v, \mu) }
  + a_h^d\del{ (u, \lambda), (v, \mu) }$.

\section{Stability and boundedness}
\label{s:consstabbound}
In this section we prove stability and boundedness of the space--time
HDG method~\cref{eq:compactwf}. Our analysis will make repeated use of
local trace and inverse inequalities valid on the finite element space
$V_h^{(p_t,p_s)}$. Using ideas from \cite{Georgoulis:2003}, the
dependence on the spatial mesh size $h_K$ and time-step $\Delta t$ is
made explicit in these inequalities. Motivated by the fact that these
two parameters differ in general, this will allow us to derive error
bounds in \cref{s:erroranalysis} that are anisotropic in $h_K$ and
$\Delta t$ as in \cite{Sudirham:2006,Sudirham:2008}. The local trace
and inverse inequalities are summarized in the following lemma.
\begin{lemma}
  \label{lem:invtrace}
  Assume that $\mathcal{K}$ is a space--time element in
  $\mathbb{R}^{d+1}$ constructed via the mappings
  $\phi_{\mathcal{K}}:\widetilde{\mathcal{K}}\to\mathcal{K}$ and
  $F_{\mathcal{K}}:\widehat{K}\to\widetilde{\mathcal{K}}$ as defined
  in~\cref{ss:approxspaces}. Assume further that the spatial shape
  regularity condition \cref{eq:shape_reg} holds. Then, for all
  $v_h \in V_{h}^{(p_{t},p_{s})}$, the following local inverse and
  trace inequalities hold:
  \begin{subequations}
    \begin{align}
      \label{eq:t_inv}
      \norm{\partial_t v_h}_{\mathcal{K}} & \le 
                                            c_{I,t} \del{\Delta t^{-1} + h_K^{-1}} \norm{v_h}_{\mathcal{K}}, 
      \\
      \label{eq:s_inv}
      \norm{\overline{\nabla}_h v_h}_{\mathcal{K}} & \le 
                                                     c_{I,s} h_K^{-1} \norm{v_h}_{\mathcal{K}}, 
      \\
      \label{eq:s_trace}
      \norm{v_h}_{\mathcal{Q}} & \le c_{T,\mathcal{Q}} 		
                                 h_K^{-\frac{1}{2}}\norm{v_h}_{\mathcal{K}},  
      \\
      \label{eq:st_trace} 
      \norm{v_h}_{\mathcal{\partial \mathcal{K}}} & \le  
                            c_{T,\partial \mathcal{K}} \del{\Delta t^{-\frac{1}{2}} + h_K^{-\frac{1}{2}}}  \norm{v_h}_{\mathcal{K}}, 
    \end{align} 
    \label{eq:invtrace}
  \end{subequations}
  where $c_{I,s}$, $c_{I,t}$, $c_{T,\mathcal{Q}}$, and
  $c_{T,\partial{\mathcal{K}}}$ are constants depending on the
  polynomial degrees $p_t$ and $p_s$, the spatial shape-regularity
  constant $c_r$, and the Jacobian of the mapping
  $\phi_{\mathcal{K}}$, but independent of the spatial mesh size $h_K$
  and the time step $\Delta t$.
\end{lemma} 
\begin{proof}
  Inequalities~\cref{eq:t_inv}--\cref{eq:st_trace} are space--time
  variants of those found in~\cite[Corollary 3.54, Corollary
  3.59]{Georgoulis:2003}.
\end{proof}

Additionally, we will require the following discrete Poincar\'{e}
inequality valid for
$(v_h,\mu_h) \in V_{h}^{\star}$~\cite{Sudirham:2006},
\begin{equation} 
  \label{eq:discpoinc}
  \norm{v_h}^2_{\mathcal{E}} \le c_p^2
  \del{ \sum_{\mathcal{K}\in\mathcal{T}_h}\norm{\overline{\nabla}_h v_h}_{\mathcal{K}}^2 
    + \sum_{\mathcal{K}\in\mathcal{T}_h} \frac{1}{h_K}\norm{v_h - \mu_h}^2_{\mathcal{Q}} },
\end{equation}
where $c_p >0$ is a constant independent of the spatial mesh size
$h_K$ and time-step $\Delta t$.

Consider the following extended function spaces on $\mathcal{E}$ and
$\Gamma$:
\begin{equation}
  V(h) = V^{(p_t,p_s)}_h + H^2(\mathcal{E}), \qquad
  M(h) = M^{(p_t,p_s)}_h + H^{3/2}(\Gamma),
\end{equation}
where $H^{3/2}(\Gamma)$ is the trace space of $H^2(\mathcal{E})$. For
notational purposes we also introduce
$V^{\star}(h) = V(h) \times M(h)$. We define three norms on
$V^{\star}(h)$. First, the ``stability'' norm is defined as
\begin{multline}
  \label{eq:hdgnorm}
  \tnorm{ (v,\mu) }_{v}^2 =
  \norm{v}_{\mathcal{E}}^2
  + \norm{ \beta_n^{1/2} \mu }^2_{\partial\mathcal{E}_N}
  + \sum_{\mathcal{K}\in\mathcal{T}_h}
  \norm{ \beta_n^{1/2}(v - \mu) }^2_{\partial \mathcal{K}} 
  \\
  + \sum_{\mathcal{K}\in\mathcal{T}_h} \nu 
  \norm{\overline{\nabla}_h v}^2_{\mathcal{K}}
  + \sum_{\mathcal{K}\in\mathcal{T}_h} \frac{\nu}{h_K}
  \norm{v - \mu}^2_{\mathcal{Q}},
\end{multline}
where for ease of notation we have defined
$\beta_n = |\beta \cdot n|$. Additionally, we introduce a stronger
norm obtained by endowing the ``stability'' norm with an additional
term controlling the $L^2$-norm of time derivatives:
\begin{equation}
  \label{eq:infsupnorm}
  \tnorm{ (v,\mu) }_s^2 = \tnorm{ (v,\mu) }_v^2 
  + \sum_{\mathcal{K}\in\mathcal{T}_h}\frac{\Delta t h_K^2}{\Delta t + h_K} 
  \norm{\partial_t v}^2_{\mathcal{K}}.
\end{equation}
To prove boundedness of the bilinear form in \cref{ss:boundedness} we
introduce the following norm:
\begin{align} 
  \label{eq:bndnorm}
  \tnorm{ (v,\mu) }_{s,\star}^2 =& \tnorm{ (v,\mu) }_{s}^2 + \sum_{\mathcal{K}\in\mathcal{T}_h} 
                                  \norm{\beta_n^{1/2} v}_{\partial \mathcal{K}^+}^2
                                  + \sum_{\mathcal{K}\in\mathcal{T}_h} 
                                  \norm{\beta_n^{1/2}\mu}_{\partial \mathcal{K}^-}^2 
  \\ \notag
                                &+ \sum_{\mathcal{K}\in\mathcal{T}_h}  h_K\nu 
                                  \norm{\overline{\nabla}_h v \cdot \bar{n}}_{\mathcal{Q}}^2
                                  +  \sum_{\mathcal{K}\in\mathcal{T}_h} 
                                  \frac{\Delta t + h_K}{\Delta t h_K^2}
                                  \norm{v}_{\mathcal{K}}^2,
\end{align}
where $\partial\mathcal{K}^+$ denotes the outflow part of the boundary
(where $\beta\cdot n > 0$) and where $\partial\mathcal{K}^-$ denotes
the inflow part of the boundary (where $\beta\cdot n \le 0$). The
additional terms are required since the inequalities
in~\cref{lem:invtrace} are valid only on the discrete space
$V_h^{(p_t,p_s)}$.

Let $u \in H^{2}(\mathcal{E})$ solve the advection--diffusion
problem~\cref{eq:modprob}. Defining the trace operator
$\gamma : H^2(\mathcal{E}) \to H^{3/2}(\Gamma)$, restricting functions
in $H^2(\mathcal{E})$ to $\Gamma$, and letting
$\boldsymbol{u} = (u, \gamma(u))$, we have
\begin{equation}
  \label{eq:consistency}
  a_h(\boldsymbol{u}, (v_h,\mu_h) ) = 
  \sum_{\mathcal{K}\in\mathcal{T}_h}\int_{\mathcal{K}}fv_h \dif x + \int_{\partial\mathcal{E}_N} g \mu_h \dif s
  \qquad \forall (v_h,\mu_h) \in V_h^{\star}.
\end{equation}
This consistency result follows by noting that $u = \gamma(u)$ on
element boundaries, integration by parts in space--time,
single-valuedness of $\beta \cdot n $,
$\overline{\nabla}_h u\cdot \bar{n}$, $u$ and $\mu_h$ on element
boundaries, the fact that $\mu_h = 0$ on $\partial\mathcal{E}_D$, and
that $u$
solves~\cref{eq:modprob}--\cref{eq:boundaryinitialconditions}. An
immediate consequence of consistency is Galerkin orthogonality: Let
$(u_h,\lambda_h) \in V_h^{\star}$ solve~\cref{eq:compactwf}, then
\begin{equation} 
  \label{eq:gal_orth}
  a_h\del{ (u,\gamma(u)) - (u_h, \lambda_h), (v_h,\mu_h) } = 0, 
  \quad \forall (v_h,\mu_h) \in V^{\star}_h.
\end{equation}
%

\subsection{Boundedness}
\label{ss:boundedness}

We now turn to the boundedness of the bilinear form.

\begin{lemma}[Boundedness] 
  \label{lem:bndness}
  There exists a $c_B>0$, independent of $h_K$ and $\Delta t$, such
  that for all $\boldsymbol{u} = (u, \gamma(u)) \in V^{\star}(h)$ and all
  $(v_h,\mu_h) \in V_h^{\star}$,
  \begin{equation}
    \left| a_h(\boldsymbol{u}, (v_h,\mu_h)) \right| \le c_B \tnorm{ \boldsymbol{u}}_{s,\star} \tnorm{ (v_h,\mu_h) }_s.
  \end{equation}
\end{lemma}
\begin{proof}
  We will begin by bounding each term of the advective part of the
  bilinear form, $a_h^a(\boldsymbol{u}, \boldsymbol{v}_h)$. We note that
  \begin{equation}
    \begin{split}
      \label{eq:adv_bilin_bnd}
      |a_h^a(\boldsymbol{u}, (v_h,\mu_h))| \le& 
      \bigg|\sum_{\mathcal{K}\in\mathcal{T}_h}\int_{\mathcal{K}} \beta u \cdot \nabla_h v_h \dif x \bigg|
      + \bigg|\int_{\partial\mathcal{E}_N}\tfrac{1}{2}\del{\beta \cdot n  
        + |\beta \cdot n |}\gamma(u)\mu_h \dif s \bigg|
      \\ &
      + \bigg| \sum_{\mathcal{K}\in\mathcal{T}_h}\int_{\partial\mathcal{K}} \tfrac{1}{2}\del{
        \beta \cdot n  (u+\gamma(u))
        + |\beta \cdot n |(u - \gamma(u))} \del{v_h - \mu_h} \dif s \bigg|.      
    \end{split}
  \end{equation}
  To obtain a bound for the first term on the right hand side of
  \cref{eq:adv_bilin_bnd}, we first recall
  $\beta \cdot \nabla_h v_h = \partial_t v_h + \bar{\beta} \cdot
  \overline{\nabla}_h v_h$, so that
  \begin{equation}
    \left|\sum_{\mathcal{K}\in\mathcal{T}_h}\int_{\mathcal{K}} \beta u \cdot \nabla_h v_h \dif x \right|
    \le \sum_{\mathcal{K}\in\mathcal{T}_h} \int_{\mathcal{K}} \left| u\partial_t v_h \right| \dif x
    + \sum_{\mathcal{K}\in\mathcal{T}_h} \int_{\mathcal{K}} \left|\bar{\beta}u \cdot\overline{\nabla}_h v_h\right| \dif x.
  \end{equation}
  Both terms on the right hand side may be bounded using the
  Cauchy--Schwarz inequality:
  \begin{align}
    \sum_{\mathcal{K}\in\mathcal{T}_h}\int_{\mathcal{K}} \left| u\partial_t v_h \right| \dif x 
    &\le \sum_{\mathcal{K}\in\mathcal{T}_h} \del{ \frac{\Delta t + h_K}{\Delta t h_K^2} }^{\frac{1}{2}} \norm{u}_{\mathcal{K}}
      \del{ \frac{\Delta t h_K^2}{\Delta t + h_K} }^{\frac{1}{2}} \norm{\partial_t v_h}_{\mathcal{K}}\\
    & \le \tnorm{ \boldsymbol{u}}_{s,\star}    \tnorm{ (v_h,\mu_h) }_s, \notag
    \\
    \sum_{\mathcal{K}\in\mathcal{T}_h} \int_{\mathcal{K}} \left|\bar{\beta}u \cdot\overline{\nabla}_h v_h\right| \dif x 
    & \le \norm{\bar{\beta}}_{L^{\infty}(\mathcal{E})}  \sum_{\mathcal{K}\in\mathcal{T}_h} \nu^{-1/2} \norm{u}_{\mathcal{K}}
      \nu^{1/2} \norm{\overline{\nabla}_h v_h}_{\mathcal{K}} \\
    & \le  \norm{\bar{\beta}}_{L^{\infty}(\mathcal{E})}\nu^{-1/2} \tnorm{ \boldsymbol{u}}_{s,\star}    \tnorm{ (v_h,\mu_h) }_s. \notag
  \end{align}
  The integral over the mixed boundary $\partial\mathcal{E}_N$
  in~\cref{eq:adv_bilin_bnd} may also be bounded via the
  Cauchy--Schwarz inequality:
  \begin{align} 
    \left| \int_{\partial\mathcal{E}_N}\tfrac{1}{2}\del{\beta \cdot n  
    + |\beta \cdot n |}\gamma(u)\mu_h \dif s \right| 
    & \le  \norm{ \beta_n^{1/2} \gamma(u) }_{\partial\mathcal{E}_N} \notag
      \norm{ \beta_n^{1/2} \mu_h }_{\partial\mathcal{E}_N} \\
    & \le \tnorm{ \boldsymbol{u}}_{s,\star}    \tnorm{ (v_h,\mu_h) }_s. \notag
  \end{align}
  For the final term appearing on the right hand side of
  \cref{eq:adv_bilin_bnd}, we have the bound
  \begin{align}
    \bigg| \sum_{\mathcal{K}\in\mathcal{T}_h} & \int_{\partial\mathcal{K}}  \tfrac{1}{2}\del{
                                                \beta \cdot n  (u+\gamma(u))
                                                + |\beta \cdot n |(u - \gamma(u))} \del{v_h - \mu_h} \dif s \bigg| \\ \notag
                                              & \le  \sum_{\mathcal{K}\in\mathcal{T}_h}\int_{\partial\mathcal{K}^{+}} \left|
                                                \beta \cdot n  \del{v_h - \mu_h} u \right| \dif s 
                                                + \sum_{\mathcal{K}\in\mathcal{T}_h}\int_{\partial\mathcal{K}^{-}} \left| 
                                                \beta \cdot n  \del{v_h - \mu_h} \gamma(u) \right| \dif s \\ \notag
                                              & \le  \sum_{\mathcal{K}\in\mathcal{T}_h} \del{ \norm{\beta_n^{1/2} u}_{\partial \mathcal{K}^+}
                                                +  \norm{\beta_n^{1/2}\gamma(u)}_{\partial \mathcal{K}^-} }
                                                \norm{ \beta_n^{1/2}(v_h - \mu_h) }_{\partial \mathcal{K}} \\ \notag
                                              & \le \sqrt{2} \tnorm{ \boldsymbol{u}}_{s,\star}    \tnorm{ (v_h,\mu_h) }_s,
  \end{align} 
  where we used the triangle inequality for the first inequality, the
  Cauchy--Schwarz inequality for the second inequality, and finally
  combined the discrete Cauchy--Schwarz inequality with the fact that
  $(a+b)^2 \le 2(a^2 + b^2)$. Collecting the above bounds we obtain,
  for all $\boldsymbol{u} \in V^{\star}(h)$ and
  $(v_h,\mu_h) \in V_h^{\star}$,
  \begin{equation} 
    \label{eq:adv_bnd}
    \left| a^a_h(\boldsymbol{u}, (v_h,\mu_h)) \right| \le c_{B,a} \tnorm{ \boldsymbol{u}}_{s,\star} \tnorm{ (v_h,\mu_h) }_s,
  \end{equation}
  where
  $c_{B,a} = 2 + \sqrt{2} +
  \norm{\bar{\beta}}_{L^{\infty}(\mathcal{E})}\nu^{-1/2}$.
  
  We now shift our focus to the diffusive part of the bilinear form,
  $a_h^d(\boldsymbol{u}, (v_h,\mu_h))$. We note that
  \begin{multline}
    \label{eq:diff_bilin_bnd}
    |a_h^d(\boldsymbol{u}, (v_h,\mu_h))|  \le
    \bigg|\sum_{\mathcal{K}\in\mathcal{T}_h}\int_{\mathcal{K}}\nu\overline{\nabla}_h u \cdot\overline{\nabla}_h v_h \dif x \bigg|
    + \bigg| \sum_{\mathcal{K}\in\mathcal{T}_h}\int_{\mathcal{Q}}\frac{\nu \alpha}{h_K}(u - \gamma(u))\del{v_h - \mu_h} \dif s \bigg|
    \\ 
    + \bigg| \sum_{\mathcal{K}\in\mathcal{T}_h}\int_{\mathcal{Q}} \sbr{ \nu(u - \gamma(u))\overline{\nabla}_h v_h \cdot \bar{n} 
      + \nu\overline{\nabla}_h u\cdot \bar{n} \del{v_h - \mu_h}}\dif s \bigg|.
  \end{multline}
  By the Cauchy--Schwarz inequality, the first two terms on the right
  hand side of \cref{eq:diff_bilin_bnd} can be bounded by
  $(1+\alpha)\tnorm{ \boldsymbol{u}}_{s,\star} \tnorm{ (v_h,\mu_h) }_s$. To
  bound the remaining term of $a_h^d(\boldsymbol{u}, (v_h,\mu_h))$, we note
  that
  \begin{align}
    \label{eq:remainingtermad}
    \bigg| \sum_{\mathcal{K}\in\mathcal{T}_h}\int_{\mathcal{Q}} & \sbr{ \nu(u - \gamma(u))\overline{\nabla}_h v_h \cdot \bar{n} + \nu\overline{\nabla}_h u\cdot \bar{n} \del{v_h - \mu_h}}\dif s \bigg| \\
                                                                & \le \sum_{\mathcal{K}\in\mathcal{T}_h}  \int_{\mathcal{Q}} \left| \nu(u - \gamma(u))\overline{\nabla}_h v_h \cdot \bar{n}\right| \dif s \notag
                                                                  + \sum_{\mathcal{K}\in\mathcal{T}_h}\int_{\mathcal{Q}} \left| \nu\overline{\nabla}_h u\cdot \bar{n} \del{v_h - \mu_h} \right|\dif s.
  \end{align}
  Application of the Cauchy--Schwarz inequality to the first term on
  the right hand side of~\cref{eq:remainingtermad}, followed by the
  trace inequality \cref{eq:s_trace}, yields
  \begin{align}
    \sum_{\mathcal{K}\in\mathcal{T}_h}  \int_{\mathcal{Q}} & \left| \nu(u - \gamma(u))\overline{\nabla}_h v_h \cdot \bar{n}\right| \dif s
    \\ \notag
    & \leq c_{T,\mathcal{Q}} \del{ \sum_{\mathcal{K}\in\mathcal{T}_h} \frac{\nu}{h_K} \norm{u - \gamma(u)}_{\mathcal{Q}}^2 }^{\frac{1}{2}}
      \del{ \sum_{\mathcal{K}\in\mathcal{T}_h} \nu \norm{\overline{\nabla}_h v_h}_{\mathcal{K}}^2 }^{\frac{1}{2}} 
    \\ \notag
    & \leq c_{T,\mathcal{Q}} \tnorm{ \boldsymbol{u} }_{s, \star}    \tnorm{ (v_h,\mu_h) }_s.
  \end{align}
  Finally, to bound the second term on the right hand side
  of~\cref{eq:remainingtermad}, we apply the Cauchy--Schwarz
  inequality:
  \begin{equation}
    \label{eq:diff_term}
    \sum_{\mathcal{K}\in\mathcal{T}_h}\int_{\mathcal{Q}} \left| \nu\overline{\nabla}_h u\cdot \bar{n} \del{v_h - \mu_h} \right|\dif s
    \le  \tnorm{ \boldsymbol{u}}_{s, \star} \tnorm{ (v_h,\mu_h) }_s.
  \end{equation}
  Therefore, for all $\boldsymbol{u} \in V^{\star}(h)$ and
  $(v_h,\mu_h) \in V_h^{\star}$,
  \begin{equation} 
    \label{eq:diff_bnd}
    \left| a^d_h(\boldsymbol{u}, (v_h,\mu_h)) \right| \le c_{B,d} \tnorm{ \boldsymbol{u}}_{s,\star} \tnorm{ (v_h,\mu_h) }_s,
  \end{equation}
  where $c_{B,d} = 2 + \alpha + c_{T,\mathcal{Q}}$. Combining
  \cref{eq:adv_bnd} with \cref{eq:diff_bnd} yields the assertion with
  $c_B = c_{B,a} + c_{B,d}$.
\end{proof}

In the sequel, we will also make use of the following bound valid for
all $(u_h, \lambda_h), (v_h,\mu_h) \in V^{\star}_h$:
\begin{equation}
  \label{eq:boundedness_ahd_Vstar}
  | a_h^d((u_h, \lambda_h), (v_h,\mu_h))| \le c_{d}\tnorm{ (u_h, \lambda_h) }_v\tnorm{ (v_h,\mu_h) }_v,
\end{equation}
which follows immediately from \cref{eq:diff_bnd} using the
equivalence of norms on finite dimensional spaces.  However, to
quantify the constant $c_{d}$ to ensure its independence of $h_K$ and
$\Delta t$, we may simply repeat the proof of the bound on
\cref{eq:diff_bilin_bnd}, instead applying the trace inequality
\cref{eq:s_trace} to the term in \cref{eq:diff_term} to obtain
$c_{d} = 1 + \alpha + 2c_{T,\mathcal{Q}}$.
%

\subsection{Stability}
\label{ss:stability}

Next we demonstrate that the method is stable in the norm
\cref{eq:hdgnorm} over the space $V_{h}^{\star}$:

\begin{lemma}[Stability]
  \label{lem:stability}
  Let $\alpha$ be the penalty parameter appearing
  in~\cref{eq:bilin_ad} which is such that
  $\alpha > c_{T,\mathcal{Q}}^2$ where $c_{T,\mathcal{Q}}$ is the
  constant from the local trace inequality \cref{eq:s_trace}. Further,
  let $c_{\alpha} = (\alpha - c_{T,\mathcal{Q}}^2)/(1+\alpha)$ and
  suppose there exists a constant $\beta_0 > 0$ such that
  \begin{equation}
    \label{eq:def_beta0}
    \frac{c_{\alpha}\nu}{c_p^2}
    +  \inf_{x\in\mathcal{E}}\overline{\nabla}_h \cdot \bar{\beta} \ge \beta_0 > 0,
  \end{equation}
  where $c_p$ is the constant from the discrete Poincar\'{e}
  inequality \cref{eq:discpoinc}.  Then there exists a constant $c_c$,
  independent of $h_K$ and $\Delta t$, such that
  \begin{equation}
    \label{eq:ah_vv_combined_final}
    a_h((v_h,\mu_h), (v_h,\mu_h)) \ge c_c \tnorm{ (v_h,\mu_h) }_v^2, 
    \qquad \forall (v_h,\mu_h) \in V_h^{\star}.
  \end{equation}
\end{lemma}
\begin{proof}
  By definition of the bilinear form $a_h^a(\cdot,\cdot)$
  in~\cref{eq:bilin_aa},
  \begin{multline}
    \label{eq:aha_vv}
      a_h^a((v_h,\mu_h), (v_h,\mu_h)) =
      \frac{1}{2}\sum_{\mathcal{K}\in\mathcal{T}_h}\int_{\mathcal{K}}v_h^2\nabla_h \cdot\beta  \dif x 
      -\sum_{\mathcal{K}\in\mathcal{T}_h}\int_{\partial\mathcal{K}}\frac{1}{2}\beta \cdot n v_h^2 \dif s
      \\
      + \int_{\partial\mathcal{E}_N}\frac{1}{2}\del{\beta \cdot n  
	+ |\beta \cdot n |}\mu_h^2 \dif s
      + \sum_{\mathcal{K}\in\mathcal{T}_h}\int_{\partial\mathcal{K}} \frac{1}{2}
      \beta \cdot n  (v_h+\mu_h)(v_h - \mu_h) \dif s
      \\
      + \sum_{\mathcal{K}\in\mathcal{T}_h}\int_{\partial\mathcal{K}} \frac{1}{2}
      |\beta \cdot n |(v_h - \mu_h)^2 \dif s,      
  \end{multline}
  where we used that
  $-2v_h\beta \cdot\nabla_h v_h = -\nabla_h \cdot(\beta v_h^2) +
  v_h^2\nabla_h \cdot\beta$
  and applied Gauss' Theorem. Expanding the fourth integral on the
  right hand side and using the fact that $\beta \cdot n $ and $\mu_h$
  are single-valued on element boundaries, and that $\mu_h = 0$ on
  $\partial\mathcal{E}_D$,~\cref{eq:aha_vv} reduces to
  \begin{multline}
    \label{eq:aha_vv_f}
    a_h^a((v_h,\mu_h), (v_h,\mu_h)) = 
    \frac{1}{2}\sum_{\mathcal{K}\in\mathcal{T}_h}\int_{\mathcal{K}}v_h^2\nabla_h \cdot\beta  \dif x 
    + \int_{\partial\mathcal{E}_N}\frac{1}{2} |\beta \cdot n | \mu_h^2 \dif s
    \\
    + \sum_{\mathcal{K}\in\mathcal{T}_h}\int_{\partial\mathcal{K}} \frac{1}{2}
    |\beta \cdot n |(v_h - \mu_h)^2 \dif s.      
  \end{multline}
  Next, by definition of the bilinear form $a_h^d(\cdot, \cdot)$
  in~\cref{eq:bilin_ad},
  \begin{multline}
    \label{eq:ahd_vv}
    a_h^d((v_h,\mu_h), (v_h,\mu_h)) = 
    \sum_{\mathcal{K}\in\mathcal{T}_h}\int_{\mathcal{K}}\nu\envert{\overline{\nabla}_h v_h}^2 \dif x 
    + \sum_{\mathcal{K}\in\mathcal{T}_h}\int_{\mathcal{Q}}\frac{\nu \alpha}{h_K}\del{v_h - \mu_h}^2 \dif s
    \\
    - \sum_{\mathcal{K}\in\mathcal{T}_h}\int_{\mathcal{Q}} 2\nu\overline{\nabla}_h v_h\cdot \bar{n} \del{v_h - \mu_h}\dif s.
  \end{multline}
  Applying the Cauchy--Schwarz inequality and the trace inequality
  \cref{eq:s_trace} to the third term on the right-hand side
  of~\cref{eq:ahd_vv},
  \begin{equation}
    \label{eq:iptermbounded}
    \envert{2\sum_{\mathcal{K}\in\mathcal{T}_h}\int_{\mathcal{Q}}\nu\overline{\nabla}_h v_h\cdot \bar{n} (v_h - \mu_h) \dif s}
    \le 2 \nu^{1/2} c_{T,\mathcal{Q}} \norm{\overline{\nabla}_h v_h}_{\mathcal{K}}
    \nu^{1/2}h_K^{-1/2}\norm{v_h-\mu_h}_{\mathcal{Q}}.
  \end{equation}
  Combining~\cref{eq:ahd_vv} and~\cref{eq:iptermbounded}, and choosing
  $\alpha > c_{T,\mathcal{Q}}^2$,
  \begin{equation}
    \label{eq:ahd_vv_b}
    \begin{split}
      a_h^d(&(v_h,\mu_h), (v_h,\mu_h)) \\
      \ge &
      \sum_{\mathcal{K}\in\mathcal{T}_h} \del{
	\nu \norm{\overline{\nabla}_h v_h}^2_{\mathcal{K}}
	- 2 c_{T,\mathcal{Q}}\nu \norm{\overline{\nabla}_h v_h}_{\mathcal{K}}h_K^{-1/2}\norm{v_h-\mu_h}_{\mathcal{Q}}
	+ \frac{\nu \alpha}{h_K}\norm{v_h - \mu_h}^2_{\mathcal{Q}} }
      \\
      \ge &
      \sum_{\mathcal{K}\in\mathcal{T}_h} 
      \frac{\alpha - c_{T,\mathcal{Q}}^2}{1+\alpha}
      \del{ \nu \norm{\overline{\nabla}_h v_h}^2_{\mathcal{K}}
        + \frac{\nu}{h_K}\norm{v_h - \mu_h}^2_{\mathcal{Q}} }.
    \end{split}
  \end{equation}
  The second inequality follows from noting that for
  $\alpha > \psi^2$, with $\psi$ a positive real number, it holds that
  $x^2 - 2 \psi xy + \alpha y^2 \ge
  (\alpha-\psi^2)(x^2+y^2)/(1+\alpha)$,
  for $x,y\in \mathbb{R}$~\cite{Pietro:book}, and taking
  $x = \nu^{1/2}\norm{\overline{\nabla}_h v_h}_{\mathcal{K}}$,
  $y = \nu^{1/2} h_K^{-1/2} \norm{v_h - \mu_h}_{ \mathcal{Q}}$ and
  $\psi = c_{T,\mathcal{Q}}$. Combining~\cref{eq:aha_vv_f}
  and~\cref{eq:ahd_vv_b}, and using that
  $\nabla_h \cdot\beta =\overline{\nabla}_h\cdot \bar{\beta}$,
  \begin{multline}
    \label{eq:ah_vv_combined}
      a_h((v_h,\mu_h), (v_h,\mu_h)) \ge
      \sum_{\mathcal{K}\in\mathcal{T}_h}
      \frac{1}{2}\int_{\mathcal{K}}v_h^2 \overline{\nabla}_h \cdot \bar{\beta} \dif x 
      + \frac{1}{2} \norm{ \beta_n^{1/2} \mu_h }^2_{\partial\mathcal{E}_N}
      \\
      + \frac{1}{2} \sum_{\mathcal{K}\in\mathcal{T}_h}\norm{ \beta_n^{1/2}(v_h - \mu_h) }^2_{\partial \mathcal{K}}
      + \sum_{\mathcal{K}\in\mathcal{T}_h} c_{\alpha} \nu \norm{\overline{\nabla}_h v_h}^2_{\mathcal{K}}
      + \sum_{\mathcal{K}\in\mathcal{T}_h} c_{\alpha}\frac{\nu}{h_K}\norm{v_h - \mu_h}^2_{\mathcal{Q}}.
  \end{multline}
  Using the discrete Poincar\'{e} inequality~\cref{eq:discpoinc} and~\cref{eq:def_beta0}, we
  obtain from~\cref{eq:ah_vv_combined}:
  \begin{multline}
    \label{eq:ah_vv_combined_newbound}
    a_h((v_h,\mu_h), (v_h,\mu_h))
    \ge \frac{1}{2}\beta_0\norm{v_h}_{\mathcal{E}}^2
    + \frac{1}{2} \norm{ \beta_n^{1/2} \mu_h }^2_{\partial\mathcal{E}_N}
    \\
    + \frac{1}{2} \sum_{\mathcal{K}\in\mathcal{T}_h}\norm{ \beta_n^{1/2}(v_h - \mu_h) }^2_{\partial \mathcal{K}}
    + \frac{1}{2}c_{\alpha} \sum_{\mathcal{K}\in\mathcal{T}_h} \nu \norm{\overline{\nabla}_h v_h}^2_{\mathcal{K}}
    + \frac{1}{2}c_{\alpha}\sum_{\mathcal{K}\in\mathcal{T}_h} \frac{\nu}{h_K}\norm{v_h - \mu_h}^2_{\mathcal{Q}}.
  \end{multline}
  The result follows with $c_c = \min(\beta_0, c_{\alpha})/2$.
\end{proof}

\subsection{The inf-sup condition} 
\label{ss:infsupstability}

Stability was proven in~\cref{ss:stability} with respect to the norm
$\tnorm{(\cdot, \cdot)}_v$. To obtain the error estimates
in~\cref{s:erroranalysis}, we instead consider a norm with additional
control over the time derivatives of the solution. For this we prove
an inf-sup condition with respect to the stronger norm
\cref{eq:infsupnorm} following ideas in~\cite{Cangiani:2017,
  Pietro:book, Wells:2011}. We first state the inf-sup condition. 

\begin{theorem}[The inf-sup condition]
  \label{thm:infsupcondition}
  There exists $c_i > 0$, independent of $h_K$ and $\Delta t$, such
  that for all $(w_h, \lambda_h) \in V_h^{\star}$
  \begin{equation}
    \label{eq:infsupcondition}
    c_i \tnorm{ (w_h, \lambda_h) }_s \le \sup_{ (v_h, \mu_h) \in V_h^{\star} } 
    \frac{ a_h((w_h, \lambda_h), (v_h, \mu_h)) }{\tnorm{ (v_h, \mu_h) }_s} \,.
  \end{equation}
\end{theorem}

The proof of the inf-sup condition follows after the following two
intermediate results.

\begin{lemma}
  \label{lem:infsupcondition_pre1}
  Let $(w_h,\lambda_h) \in V_h^{\star}$ and let
  $z_h = \frac{\Delta t h_K^2}{\Delta t + h_K} \partial_t w_h$. There
  exists a $c_1 > 0$, independent of $h_K$ and $\Delta t$, such that
  \begin{equation*}
    \label{eq:tnormbound}
    \tnorm{ (z_h, 0) }_s \le c_1 \tnorm{ (w_h, \lambda_h) }_s \,.
  \end{equation*}
\end{lemma}
\begin{proof}
  We bound each component of $\tnorm{ (z_h, 0) }_s$ term-by-term.
  Using the inverse inequality \cref{eq:t_inv} and that $h_K < 1$, we
  have
  \begin{equation*}
    \label{eq:tnorm_zh_1}
    \norm{z_h}_{\mathcal{E}}^2 = \sum_{\mathcal{K}\in\mathcal{T}_h} 
    \del{ \frac{\Delta t h_K^2}{\Delta t + h_K} }^2
    \norm{ \partial_t w_h}_{\mathcal{K}}^2  
    \le  c_{I,t}^2 \norm{w_h}_{\mathcal{E}}^2.
  \end{equation*}
  Similarly, the inverse inequality~\cref{eq:t_inv} and $h_K < 1$
  yields
  \begin{equation*}
    \sum_{\mathcal{K}\in\mathcal{T}_h} \nu \norm{\overline{\nabla}_h z_h}^2_{\mathcal{K}} =
    \sum_{\mathcal{K}\in\mathcal{T}_h} \nu \del{ \frac{\Delta t h_K^2}{\Delta t + h_K} }^2
     \norm{\partial_t(\overline{\nabla}_h w_h)}^2_{\mathcal{K}} 
    \le c_{I,t}^2\sum_{\mathcal{K}\in\mathcal{T}_h} \nu \norm{\overline{\nabla}_h w_h}^2_{\mathcal{K}}.    
  \end{equation*}
  Next, the facet term arising from the advective portion of the norm
  may be bounded using the trace inequality \cref{eq:st_trace}:
  \begin{align*}
    \label{eq:tnorm_zh_2}
    \sum_{\mathcal{K}\in\mathcal{T}_h}\norm{ \beta_n^{1/2}z_h}^2_{\partial \mathcal{K}}
    & \le \norm{\beta}_{L^\infty(\mathcal{E})}
      \sum_{\mathcal{K}\in\mathcal{T}_h}
      \norm{z_h}^2_{\partial \mathcal{K}} \\
    & \le c_{T,\partial \mathcal{K}}^2 \norm{\beta}_{L^\infty(\mathcal{E})}
      \sum_{\mathcal{K}\in\mathcal{T}_h} 
      \del{ \frac{\Delta t h_K^2}{\Delta t + h_K} } \norm{\partial_t w_h}^2_{\mathcal{K}}.
  \end{align*}
  The facet term diffusive portion of the norm may be bounded with an
  application of~\cref{eq:s_trace} and~\cref{eq:t_inv}:
  \begin{equation*}
    \label{eq:tnorm_zh_4}
    \sum_{\mathcal{K}\in\mathcal{T}_h} \frac{\nu}{h_K}\norm{z_h}^2_{\mathcal{Q}} 
    = 
    \sum_{\mathcal{K}\in\mathcal{T}_h} \frac{\nu}{h_K}\del{ \frac{\Delta t h_K^2}{\Delta t + h_K} }^2 \norm{\partial_tw_h}^2_{\mathcal{Q}}
    \le 
    c_{T,\mathcal{Q}}^2c_{I,t}^2\sum_{\mathcal{K}\in\mathcal{T}_h} \nu\norm{w_h}^2_{\mathcal{K}}.
  \end{equation*}
  For the remaining term, \cref{eq:t_inv} yields
  \begin{equation*}
    \label{eq:tnorm_zh_5}
    \sum_{\mathcal{K}\in\mathcal{T}_h}\frac{\Delta t h_K^2}{\Delta t + h_K}\norm{\partial_t z_h}^2_{\mathcal{K}}
    \le c_{I,t}^2 \sum_{\mathcal{K}\in\mathcal{T}_h}  \del{\frac{\Delta t h_K^2}{\Delta t + h_K}\norm{ \partial_t w_h}^2_{\mathcal{K}}}.
  \end{equation*}
  Collecting the above bounds, we obtain \cref{eq:tnormbound}, with
  $c_1 = 3c_{I,t}^2 + c_{T,\partial
    \mathcal{K}}^2\norm{\beta}_{L^{\infty}(\mathcal{E})} +
  c_{T,\mathcal{Q}}^2c_{I,t}^2$.
\end{proof}
\begin{lemma}
  \label{lem:infsupcondition_pre2}
  Let $(w_h,\lambda_h) \in V_h^{\star}$ and let
  $z_h = \frac{\Delta t h_K^2}{\Delta t + h_K} \partial_t w_h$. There
  exists a $c_2 > 0$, independent of $h_K$ and $\Delta t$, such that
  if $(v_h, \mu_h) = c_2(w_h, \lambda_h) + (z_h, 0) \in V_h^{\star}$,
  then
  \begin{equation}
    \label{eq:stability_in_vw}
    \frac{1}{2}\tnorm{ (w_h, \lambda_h) }_s^2 \le a_h((w_h, \lambda_h), (v_h, \mu_h)).
  \end{equation}
\end{lemma}

\begin{proof}
  Note that
  $a_h((w_h, \lambda_h), (z_h, 0))= a_h^a((w_h, \lambda_h), (z_h, 0))+
  a_h^d((w_h, \lambda_h), (z_h, 0))$.
  Integrating by parts the volume integral of $a_h^a(\cdot,\cdot)$ we
  have the following decomposition:
  \begin{multline} 
    \label{eq:dt_term}
    \sum_{\mathcal{K}\in\mathcal{T}_h}\frac{\Delta t h_K^2}{\Delta t + h_K} \norm{\partial_t w_h}^2_{\mathcal{K}}
    = a_h((w_h, \lambda_h), (z_h, 0)) - a_h^d((w_h, \lambda_h), (z_h, 0)) 
    \\
    - \sum_{\mathcal{K} \in \mathcal{T}_h} \frac{\Delta t h_K^2}{\Delta t+h_K} \int_{\mathcal{K}}
    w_h \overline{\nabla}_h \cdot \bar{\beta} \partial_t w_h \dif x 
    - \sum_{\mathcal{K} \in \mathcal{T}_h} \frac{\Delta t h_K^2}{\Delta t+h_K} \int_{\mathcal{K}}
     \bar{\beta} \cdot \overline{\nabla}_h w_h \partial_t w_h \dif x 
    \\
    + \frac{1}{2}\sum_{\mathcal{K} \in \mathcal{T}_h} 
    \frac{\Delta t h_K^2}{\Delta t + h_K} \int_{\partial \mathcal{K}}
    \del{\beta \cdot n - |\beta \cdot n|}\del{w_h - \lambda_h} \partial_t w_h \dif s.
  \end{multline} 
  From the boundedness of the diffusive part of the bilinear
  form~\cref{eq:boundedness_ahd_Vstar}, and application of Young's
  inequality, with $\epsilon_1 > 0$, we obtain the following bound for
  the second term on the right hand side of \cref{eq:dt_term}:
  \begin{align*}
    | a_h^d((w_h, \lambda_h), (z_h, 0)) | & \le \frac{c_d}{2\epsilon_1} \tnorm{ (z_h,0) }_v^2 + \frac{c_d \epsilon_1}{2} \tnorm{ (w_h, \lambda_h) }_v^2 
    \\
                                     &\le \frac{c_dc_1^2}{2\epsilon_1} \tnorm{ (w_h, \lambda_h) }_s^2 + \frac{c_d \epsilon_1}{2} \tnorm{ (w_h, \lambda_h) }_v^2 \\ 
                                    & \le  \frac{c_dc_1^2}{2\epsilon_1} \sum_{\mathcal{K}\in\mathcal{T}_h}\frac{\Delta t h_K^2}{\Delta t + h_K} \norm{\partial_t w_h}^2_{\mathcal{K}} +
                                      \del{ \frac{c_dc_1^2}{2\epsilon_1} + \frac{c_d \epsilon_1}{2} }\tnorm{ (w_h, \lambda_h) }_v^2,
  \end{align*}
  where we have used the fact that
  $\tnorm{\cdot}_v \le \tnorm{\cdot}_s$ and applied
  \cref{lem:infsupcondition_pre1} in the second inequality, and the
  definition of $\tnorm{\cdot}_s$ in the third inequality.  Next, to
  bound the third term on the right hand side of \cref{eq:dt_term} we
  apply the Cauchy--Schwarz inequality and~\cref{eq:t_inv} to obtain
  \begin{equation*}
    \bigg| \sum_{\mathcal{K} \in \mathcal{T}_h} \frac{\Delta t h_K^2}{\Delta t+h_K} \int_{\mathcal{K}}
    w_h \overline{\nabla}_h \cdot \bar{\beta} \partial_t w_h \dif x \bigg|
    \le c_{I,t} \norm{\overline{\nabla}_h \cdot \bar{\beta}}_{L^{\infty}(\mathcal{E})} 
    \sum_{\mathcal{K} \in \mathcal{T}_h} h_K \norm{w_h}^2_{\mathcal{K}}.
  \end{equation*}
  As for the fourth term on the right hand side of \cref{eq:dt_term},
  we first apply the Cauchy--Schwarz inequality, Young's inequality
  with some $\epsilon_2 > 0$ and~\cref{eq:s_inv},
  \begin{multline*}
    \bigg| \sum_{\mathcal{K}\in\mathcal{T}_h}\frac{\Delta t h_K^2}{\Delta t + h_K}
    \int_{\mathcal{K}} \bar{\beta} \cdot \overline{\nabla}_h w_h \partial_t w_h\dif x \bigg| 
    \le \sum_{\mathcal{K}\in\mathcal{T}_h}\frac{\Delta t h_K^2}{\Delta t + h_K} \norm{\bar{\beta}}_{L^{\infty}(\mathcal{K})} 
      \norm{\partial_t w_h}_{\mathcal{K}} \norm{\overline{\nabla}_h w_h}_{\mathcal{K}} \\
    \le \norm{\bar{\beta}}_{L^{\infty}(\mathcal{E})}\sbr{
      \sum_{\mathcal{K}\in\mathcal{T}_h}\frac{c_{I,s}^2}{2\epsilon_2}\norm{w_h}_{\mathcal{K}}^2 
      + \sum_{\mathcal{K}\in\mathcal{T}_h} \frac{\epsilon_2}{2}\del{\frac{\Delta t h_K^2}{\Delta t + h_K}}
      \norm{\partial_t w_h}^2_{\mathcal{K}}}.
  \end{multline*}
  For the remaining term on the right hand side of \cref{eq:dt_term},
  we use the Cauchy--Schwarz inequality, Young's inequality with some
  $\epsilon_3 > 0$, and apply the trace inequality~\cref{eq:st_trace}
  to find:
  \begin{multline}
    \bigg| \sum_{\mathcal{K}\in\mathcal{T}_h}\frac{\Delta t h_K^2}{\Delta t + h_K} 
    \int_{\partial\mathcal{K}} \frac{1}{2}\del{ \beta \cdot n 
      - |\beta \cdot n |}\partial_t w_h \del{w_h - \lambda_h} \dif s \bigg| 
    \\
    \le
    \frac{c_{T,\partial \mathcal{K}}^2}{2\epsilon_3}\sum_{\mathcal{K}\in\mathcal{T}_h}  
    h_K\frac{\Delta t h_K^2}{\Delta t + h_K} \norm{\partial_t w_h}_{\mathcal{K}}^2+
    \frac{\epsilon_3}{2}\norm{\beta}_{L^{\infty}(\mathcal{E})}
    \sum_{\mathcal{K}\in\mathcal{T}_h} \norm{\beta_n^{1/2} (w_h-\lambda_h)}_{\partial \mathcal{K}}^2.
  \end{multline}
  Combining all of the above estimates,
  \begin{multline}
    \sum_{\mathcal{K}\in\mathcal{T}_h} \frac{\Delta t h_K^2}{\Delta t + h_K} \norm{\partial_t w_h}^2_{\mathcal{K}}
    \le a_h((w_h, \lambda_h), (z_h,0) ) \\
    + 
    \del{\frac{c_dc_1^2}{2\epsilon_1} + \frac{\epsilon_2}{2}\norm{\bar{\beta}}_{L^{\infty}(\mathcal{E})} 
      + \frac{c_{T,\partial\mathcal{K}}^2}{2\epsilon_3}} \sum_{\mathcal{K}\in\mathcal{T}_h}\frac{\Delta t h_K^2}{\Delta t + h_K} 
    \norm{\partial_t w_h}^2_{\mathcal{K}} 
    \\
    + \bigg(\frac{c_dc_1^2}{2\epsilon_1} + \frac{c_d \epsilon_1}{2}
    + c_{I,t} \norm{\overline{\nabla}_h \cdot \bar{\beta}}_{L^{\infty}(\mathcal{E})}
      +\frac{c_{I,s}^2}{2\epsilon_2}\norm{\bar{\beta}}_{L^{\infty}(\mathcal{E})}
      \\
      +
      \frac{\epsilon_3}{2}\norm{\beta}_{L^{\infty}(\mathcal{E})}     
    \bigg) 
    \tnorm{ (w_h, \lambda_h) }_v^2.
  \end{multline}
  Choosing $\epsilon_1 = 2c_dc_1^2$,
  $\epsilon_2 =
  1/(4\norm{\bar{\beta}}_{L^{\infty}(\mathcal{E})})$,
  and $\epsilon_3 = 4c_{T,\partial \mathcal{K}}^2$, adding
  $\frac{1}{2}\tnorm{ (w_h, \lambda_h) }_v^2$ to both sides, and
  rearranging yields
  \begin{multline}
    \tfrac{1}{2}\tnorm{ (w_h, \lambda_h) }_s^2
    \le a_h((w_h, \lambda_h), (z_h, 0))
    \\
    + \bigg(\tfrac{3}{4} + c_d^2c_1^2
      + c_{I,t} \norm{\overline{\nabla}_h \cdot \bar{\beta}}_{L^{\infty}(\mathcal{E})}
      + 2c_{I,s}^2\norm{\bar{\beta}}_{L^{\infty}(\mathcal{E})}^2
      + 2c_{T,\partial \mathcal{K}}^2 \norm{\beta}_{L^{\infty}(\mathcal{E})}
      \bigg)\tnorm{ (w_h, \lambda_h) }_v^2.
  \end{multline}
  From the stability of $a_h(\cdot, \cdot)$, \cref{lem:stability}, we
  have the bound
  \begin{equation}
    \tfrac{1}{2} \tnorm{ (w_h, \lambda_h) }_s^2
    \le a_h((w_h, \lambda_h), (z_h, 0)) + c_2 a_h((w_h, \lambda_h), (w_h, \lambda_h)) 
  \end{equation}
  where
  $c_2 = c_c^{-1}(\tfrac{3}{4} + c_d^2c_1^2 + c_{I,t}
  \norm{\overline{\nabla}_h \cdot
    \bar{\beta}}_{L^{\infty}(\mathcal{E})} +
  2c_{I,s}^2\norm{\bar{\beta}}_{L^{\infty}(\mathcal{E})}^2 +
  2c_{T,\partial \mathcal{K}}^2
  \norm{\beta}_{L^{\infty}(\mathcal{E})})$. The result follows.
\end{proof}

Combining \cref{lem:infsupcondition_pre1} and
\cref{lem:infsupcondition_pre2} now yields the proof for the inf-sup
condition stated in~\cref{thm:infsupcondition}.
\begin{proof}[Proof of~\cref{thm:infsupcondition}]
  Given any $(w_h, \lambda_h) \in V_h^{\star}$, consider the linear
  combination $(v_h, \mu_h) = c_2(w_h, \lambda_h) + (z_h, 0)$, with
  $z_h = \frac{\Delta t h_K^2}{\Delta t + h_K} \partial_t w_h$ and
  $c_2$ the constant from~\cref{lem:infsupcondition_pre2}. An
  application of the triangle inequality and the combination
  of~\cref{lem:infsupcondition_pre1}
  and~\cref{lem:infsupcondition_pre2} yields
  \begin{align*}
    \label{eq:towardsinfsup}
    \tnorm{ (v_h, \mu_h) }_s \tnorm{ (w_h, \lambda_h) }_s 
    &\le \tnorm{ (z_h, 0) }_s \tnorm{ (w_h, \lambda_h) }_s + c_2\tnorm{ (w_h, \lambda_h) }_s^2\\
    &\le (c_1 + c_2)\tnorm{ (w_h, \lambda_h) }_s^2 \\
    &\le 2(c_1 + c_2)a_h((w_h, \lambda_h), (v_h, \mu_h)),
  \end{align*}
  which implies the inf-sup condition with
  $c_i = \frac{1}{2}(c_1 + c_2)^{-1}$.
\end{proof}

\section{Error analysis}
\label{s:erroranalysis}

We now turn to the error analysis of the space--time HDG method. The
following C\'{e}a-like lemma will prove useful in obtaining the global
error estimate in~\cref{lem:globErrorEstimate}.

\begin{lemma}[Convergence] 
  \label{lem:convergence} 
  If
  $\boldsymbol{u} = (u,\gamma(u)) \in H^2(\mathcal{E})\times
  H^{3/2}(\Gamma)$,
  where $u$ solves~\cref{eq:advdif_tdepdomain}, and
  $(u_h, \lambda_h) \in V_h^{\star}$ is the solution to
  the discrete problem~\cref{eq:compactwf}, then
  \begin{equation}
    \tnorm{ \boldsymbol{u} - (u_h, \lambda_h) }_{s} \le \del{ 1 + \frac{c_B}{c_i} }
    \inf_{(v_h, \mu_h) \in V^{\star}_h} \tnorm{ \boldsymbol{u} - (v_h, \mu_h) }_{s,\star}\,.
  \end{equation}
\end{lemma}
\begin{proof}
  From inf-sup stability (\cref{thm:infsupcondition}), Galerkin
  orthogonality \cref{eq:gal_orth}, and boundedness
  (\cref{lem:bndness}), we have for any
  $(w_h, \omega_h) \in V_h^{\star}$
  \begin{align*}
    c_i \tnorm{ (u_h, \lambda_h) - (w_h, \omega_h) }_{s} 
    & \le \sup_{(v_h, \mu_h) \in V^{\star}_h} \frac{ a_h((u_h, \lambda_h) - (w_h, \omega_h), (v_h, \mu_h)) }{\tnorm{ (v_h, \mu_h) }_s}
    \\
    & = \sup_{(v_h, \mu_h) \in V^{\star}_h} \frac{ a_h(\boldsymbol{u} - (w_h, \omega_h), (v_h, \mu_h)) }{\tnorm{ (v_h, \mu_h) }_s} 
    \\ 
    & \le c_B \sup_{(v_h, \mu_h) \in V^{\star}_h} \frac{ \tnorm{\boldsymbol{u} - (w_h, \omega_h)}_{s,\star} \tnorm{ (v_h, \mu_h) }_s }{\tnorm{ (v_h, \mu_h) }_s} \\
    & = c_B \tnorm{\boldsymbol{u} - (w_h, \omega_h)}_{s,\star}.
  \end{align*}
  The result follows after application of the triangle inequality to
  $\tnorm{ \boldsymbol{u} - (u_h, \lambda_h) }_{s}$.
\end{proof}

We next define the projections
$\mathcal{P}: L^2(\mathcal{E}) \to V^{(p_t,p_s)}_h$ and
$\mathcal{P}^{\partial}: L^2(\Gamma) \to M^{(p_t,p_s)}_h$ which
satisfy
\begin{align} 
  & \sum_{\mathcal{K} \in \mathcal{T}_h} \int_{\mathcal{K}} \del{ w-\mathcal{P}w } v_h \dif x = 0, 
    \qquad \forall \; v_h \in V^{(p_t,p_s)}_h, 
  \\
  & \sum_{\mathcal{S} \in \mathcal{F}} \int_{\mathcal{S}} \del{ \lambda-\mathcal{P}^{\partial}\lambda } 
    \mu_h \dif s = 0, \qquad \forall \; \mu_h \in M^{(p_t,p_s)}_h.
\end{align}
These projections will be used to obtain interpolation estimates.

\begin{lemma}[Interpolation estimates]
  \label{lem:interpolation}
  Assume that $\mathcal{K}$ is a space--time element in
  $\mathbb{R}^{d+1}$ constructed via two mappings $\phi_{\mathcal{K}}$
  and $F_{\mathcal{K}}$, with
  $F_{\mathcal{K}} : \widehat{\mathcal{K}} \to
  \widetilde{\mathcal{K}}$
  and $\phi_{\mathcal{K}}: \widetilde{\mathcal{K}} \to \mathcal{K}$.
  Assume that the spatial shape-regularity condition
  \cref{eq:shape_reg} holds. Suppose
  $u|_{\mathcal{K}} \in H^{(p_t+1,p_s+1)}(\mathcal{K})$
  solves~\cref{eq:modprob}--\cref{eq:boundaryinitialconditions}. Then,
  the error $u-\mathcal{P}u$, its trace at the boundary
  $\partial \mathcal{K}$, and the error $u-\mathcal{P}^{\partial}u$ on
  $\partial \mathcal{K}$ satisfy the following error bounds:
  \begin{align}
    \label{eq:interp_u}
    &\norm{u - \mathcal{P}u}_{\mathcal{K}}^2 
      \le c \del{ h_K^{2p_s + 2} + \Delta t^{2p_t+2} } \norm{ u }^2_{H^{(p_t + 1, p_s + 1)}(\mathcal{K})},
    \\
    \label{eq:interp_grad}
    &\norm{\overline{\nabla}_h(u - \mathcal{P}u)}_{\mathcal{K}}^2
      \le c  \del{ h_K^{2p_s} + \Delta t^{2p_t+2} } \norm{u}^2_{H^{(p_t + 1, p_s + 1)}(\mathcal{K})} ,
    \\
    \label{eq:interp_dt}
    &\norm{\partial_{t} (u - \mathcal{P}u)}_{\mathcal{K}}^2
      \le c  \del{ h_K^{2p_s} + \Delta t^{2p_t} } \norm{u}^2_{H^{(p_t + 1, p_s + 1)}(\mathcal{K})},       
    \\
    \label{eq:interp_grad_trace}
    & \norm{\overline{\nabla}_h(u - \mathcal{P}u)\cdot \bar{n}}_{\mathcal{Q}}^2 
      \le c  \del{ h_K^{2p_s-1} + h_K^{-1} \Delta t^{2p_t+2} } \norm{u}^2_{H^{(p_t + 1, p_s + 1)}(\mathcal{K})}, 
    \\
    \label{eq:interp_trace}
    &\norm{u - \mathcal{P}u}_{\partial \mathcal{K}}^2 
      \le c  \del{ h_K^{2p_s + 1} + \Delta t^{2p_t+1} } \norm{u}^2_{H^{(p_t + 1, p_s + 1)}(\mathcal{K})}, 
    \\
    \label{eq:interp_facet}
    & \norm{u - \mathcal{P}^{\partial}\gamma(u)}_{\partial \mathcal{K}}
       \le c \del{ h_K^{2p_s + 1} + \Delta t^{2p_t+1} } \norm{u}^2_{H^{(p_t + 1, p_s + 1)}(\mathcal{K})}, 
  \end{align}
  where $c$ depends only on the spatial dimension $d$, the polynomial
  degrees $p_t$ and $p_s$, the spatial shape-regularity constant
  $c_r$, and the Jacobian of the mapping $\phi_{\mathcal{K}}$.
\end{lemma}
\begin{proof}
  The bounds \cref{eq:interp_u}, \cref{eq:interp_grad} and
  \cref{eq:interp_trace} have been obtained previously in~\cite[Lemma
  6.1 and Remark 6.2]{Sudirham:2006} by generalizing~\cite[Lemmas 3.13
  and 3.17]{Georgoulis:2003} to higher dimensions.  We relax the
  assumption in~\cite[Remark 6.2]{Sudirham:2006} that all spatial edge
  lengths are equal through the spatial shape-regularity assumption
  \cref{eq:shape_reg}. In doing so, the bound \cref{eq:interp_dt} may
  be obtained in an identical fashion to~\cref{eq:interp_grad}. The
  bound \cref{eq:interp_grad_trace} is obtained as follows: we derive
  a bound for the spatial derivative of the interpolation error over
  each face $\partial \mathcal{K}_i$, where $i=1,\dots,d$,
  generalizing~\cite[Lemma 3.20]{Georgoulis:2003} to the space--time
  setting. Then, summing over the faces $i=1,\dots,d$ we obtain a
  bound of the spatial derivatives of the interpolation error over
  $\mathcal{Q}=\partial\mathcal{K}\backslash(K^n\cup K^{n+1})$, and
  sum over all of the spatial derivatives to obtain the
  result. Lastly, the bound \cref{eq:interp_facet} may be inferred
  from the bound \cref{eq:interp_trace} by the optimality of the
  $L^2$-projection $\mathcal{P}^{\partial}$ on facets.
\end{proof}

With the interpolation estimates in place, we can now derive an error
bound in the $\tnorm{\cdot}_{s}$ norm:

\begin{lemma}[Global error estimate]
  \label{lem:globErrorEstimate}
  Suppose that $\mathcal{K}$ is a space--time element in
  $\mathbb{R}^{d+1}$ constructed via two mappings $\phi_{\mathcal{K}}$
  and $F_{\mathcal{K}}$, with
  $F_{\mathcal{K}} : \widehat{\mathcal{K}} \to
  \widetilde{\mathcal{K}}$
  and $\phi_{\mathcal{K}}: \widetilde{\mathcal{K}} \to \mathcal{K}$,
  and that the spatial shape-regularity condition~\cref{eq:shape_reg}
  holds. Let $\boldsymbol{u} = (u, \gamma(u))$, where
  $u|_{\mathcal{K}} \in H^{(p_t + 1, p_s + 1)}(\mathcal{K})$ solves
  the advection--diffusion problem~\cref{eq:modprob}, and where
  $\gamma(u)$ denotes the trace of $u$ on $\partial\mathcal{K}$.
  Furthermore, let $(u_h, \lambda_h) \in V^{\star}_h$ be the solution
  to the discrete problem~\cref{eq:compactwf}. Then, the following
  error bound holds:
  \begin{multline}
    \tnorm{\boldsymbol{u}- (u_h, \lambda_h)}_{s}^2 \le \\
    C \del{ h^{2p_s} + \Delta t^{2p_t+1}
    + \nu \del{ h^{2p_s} + h^{-1} \Delta t^{2p_t+1}} }
  \norm{ u }^2_{H^{(p_t + 1, p_s + 1)}(\mathcal{E})},
  \end{multline}
  where $h=max_{\mathcal{K} \in \mathcal{T}_h}h_K$ is the spatial mesh
  size $\Delta t$ is the time-step and $C>0$ a constant.
\end{lemma}
\begin{proof}
  By \cref{lem:convergence}, we may bound the discretization error
  $\boldsymbol{u} - (u_h, \lambda_h)$ in the $\tnorm{\cdot}_s$ norm by the
  interpolation error
  $ \boldsymbol{u} - (\mathcal{P}u, \mathcal{P}^{\partial}\gamma(u))$ in the
  $\tnorm{\cdot}_{s,\star}$ norm:
  \begin{equation} 
    \label{eq:disc_bnd}
    \tnorm{ \boldsymbol{u} - (u_h, \lambda_h) }_{s} \le \del{ 1 + \frac{c_B}{c_i} }
    \tnorm{ \boldsymbol{u} - (\mathcal{P}u, \mathcal{P}^{\partial}\gamma(u))}_{s,\star}.
  \end{equation}
  Thus, it suffices to bound each term of
  $\tnorm{ \boldsymbol{u} - (\mathcal{P}u,
    \mathcal{P}^{\partial}\gamma(u))}_{s,\star}$
  using the interpolation estimates in~\cref{lem:interpolation}.

  First, combining the terms involving
  $\norm{u-\mathcal{P}u}_{\mathcal{K}}$, applying \cref{eq:interp_u},
  and collecting the leading order terms,
  \begin{equation}
    \del{1+ \frac{\Delta t + h_K}{\Delta t h_K^2} } \norm{u-\mathcal{P}u}_{\mathcal{K}}^2 \le 
    c \del{ h_K^{2p_s} + \Delta t^{2p_t+2}h_K^{-2}} \norm{ u }^2_{H^{(p_t + 1, p_s + 1)}(\mathcal{K})}.
  \end{equation}
  Using the fact that
  $\tfrac{\Delta t h_K^2}{\Delta t + h_K} \le \Delta th_K$ and applying the
  estimate \cref{eq:interp_dt}, we have
  \begin{align}
    \frac{\Delta t h_K^2}{\Delta t + h_K} \norm{\partial_t (u-\mathcal{P}u)}^2_{\mathcal{K}}
    \le c\del{h_K^{2p_s+1} \Delta t + h_K\Delta t^{2p_t+1}} \norm{u}^2_{H^{(p_t + 1, p_s + 1)}(\mathcal{K})} .
  \end{align}
  Next, an application of \cref{eq:interp_grad} yields
  \begin{align}
    \nu \norm{\overline{\nabla}_h(u-\mathcal{P}u)}^2_{\mathcal{K}} 
    \le  c\nu \del{ h_K^{2p_s} + \Delta t^{2p_t+2} } \norm{u}^2_{H^{(p_t + 1, p_s + 1)}(\mathcal{K})}.
  \end{align}
  Using the triangle inequality, \cref{eq:interp_trace}, and
  \cref{eq:interp_facet}, all of the advective facet terms may be
  bounded as follows:
  \begin{multline}
    \sum_{\mathcal{K}\in\mathcal{T}_h}\del{\norm{\beta_n^{1/2} (u-\mathcal{P}u)}_{\partial \mathcal{K}}^2
    + \; \norm{\beta_n^{1/2}(u-\mathcal{P}^{\partial} u)}_{\partial \mathcal{K}}^2}
    \le \\
    c \norm{\beta}_{L^{\infty}(\mathcal{E})}\sum_{\mathcal{K}\in\mathcal{T}_h}
      \del{h_K^{2p_s + 1} + \Delta t^{2p_t+1}} \norm{u}^2_{H^{(p_t + 1, p_s + 1)}(\mathcal{K})}.
  \end{multline}
  For the diffusive facet term, we again apply the triangle
  inequality, \cref{eq:interp_trace}, and \cref{eq:interp_facet} to
  obtain
  \begin{equation}
    \frac{\nu}{h_K} \norm{u - \mathcal{P} u}^2_{\partial \mathcal{K}}
    + \frac{\nu}{h_K} \norm{u - \mathcal{P}^{\partial} u}^2_{\partial \mathcal{K}} 
    \le c\nu \del{ h_K^{2p_s} + h_K^{-1}\Delta t^{2p_t+1} } \norm{u}^2_{H^{(p_t + 1, p_s + 1)}(\mathcal{K})} .
  \end{equation}
  Lastly, applying \cref{eq:interp_grad_trace},
  \begin{equation}
    h_K\nu \norm{\overline{\nabla}_h (u-\mathcal{P}u)\cdot \bar{n}}_{\mathcal{Q}}^2
    \le  c \nu \del{ h_K^{2p_s} + \Delta t^{2p_t+2} } \norm{u}^2_{H^{(p_t + 1, p_s + 1)}(\mathcal{K})}.
  \end{equation}
  Summing over all $\mathcal{K} \in \mathcal{T}_h$, collecting all of
  the above estimates, and returning to \cref{eq:disc_bnd} yields the
  assertion.
\end{proof}

\section{Numerical example}
\label{s:numex}

In this section we validate the results of the previous sections. For
this we consider the rotating Gaussian pulse test case on a
time-dependent domain as introduced in~\cite[Section
4.3]{Rhebergen:2013}. We
solve~\cref{eq:modprob}--\cref{eq:boundaryinitialconditions} with
$\bar{\beta} = (-4x_2, 4x_1)^T$ and $f = 0$. The boundary and initial
conditions are set such that the exact solution is given by
\begin{equation}
  \label{eq:example1}
  u(t, x_1, x_2) = \frac{\sigma^2}{\sigma^2 + 2 \nu t}\exp\left(-\frac{(\tilde{x}_1 - x_{1c})^2 + (\tilde{x}_2 - x_{2c})^2}
    {2\sigma^2 + 4 \nu t}\right),
\end{equation}
where $\tilde{x}_1 = x_1 \cos(4t) + x_2 \sin(4t)$,
$\tilde{x}_2 = -x_1 \sin(4t) + x_2 \cos(4t)$,
$(x_{1c}, x_{2c}) = (-0.2, 0.1)$. Furthermore, we set $\sigma =
0.1$.

The advection--diffusion problem is solved on a time-dependent
domain. The deformation is based on a transformation of a uniform
space--time mesh $(t, x_1^0, x_2^0) \in [0, t_N] \times [-0.5, 0.5]^2$
given by
\begin{equation}
  \label{eq:meshdeformation}
  x_i = x_i^0 + A \left(\tfrac{1}{2} - x_i^0\right)\sin \left(2\pi\left(\tfrac{1}{2} - x_i^* + t\right)\right) \qquad  i = 1,2,  
\end{equation}
where and $(x_1^*, x_2^*) = (x_2, x_1)$ and $A = 0.1$. We take
$t_N=1$.

This example was implemented using the Modular Finite Element Methods
(MFEM) library~\cite{mfem-library} on unstructured hexahedral
space--time meshes. The solution on the time-dependent domain is shown
at different points in time in~\cref{fig:mesh_movement}.

In~\cref{tab:ex1_cfl1} we compute the rates of convergence in the
$\tnorm{(\cdot, \cdot)}_s$ norm using polynomial degree
$p = p_t = p_s = 1, 2, 3$. We consider both $\nu = 10^{-2}$ and
$\nu = 10^{-6}$. Mesh refinement is done simultaneously in space and
time. For the case that $\nu = 10^{-2}$ we obtain rates of convergence
of approximately $p$, as expected from~\cref{lem:globErrorEstimate},
while for $\nu = 10^{-6}$ we obtain slightly better rates of
convergence, namely $p + 1/2$.

\begin{figure}[tbp]
  \begin{center}
    \includegraphics[width=.3\linewidth]{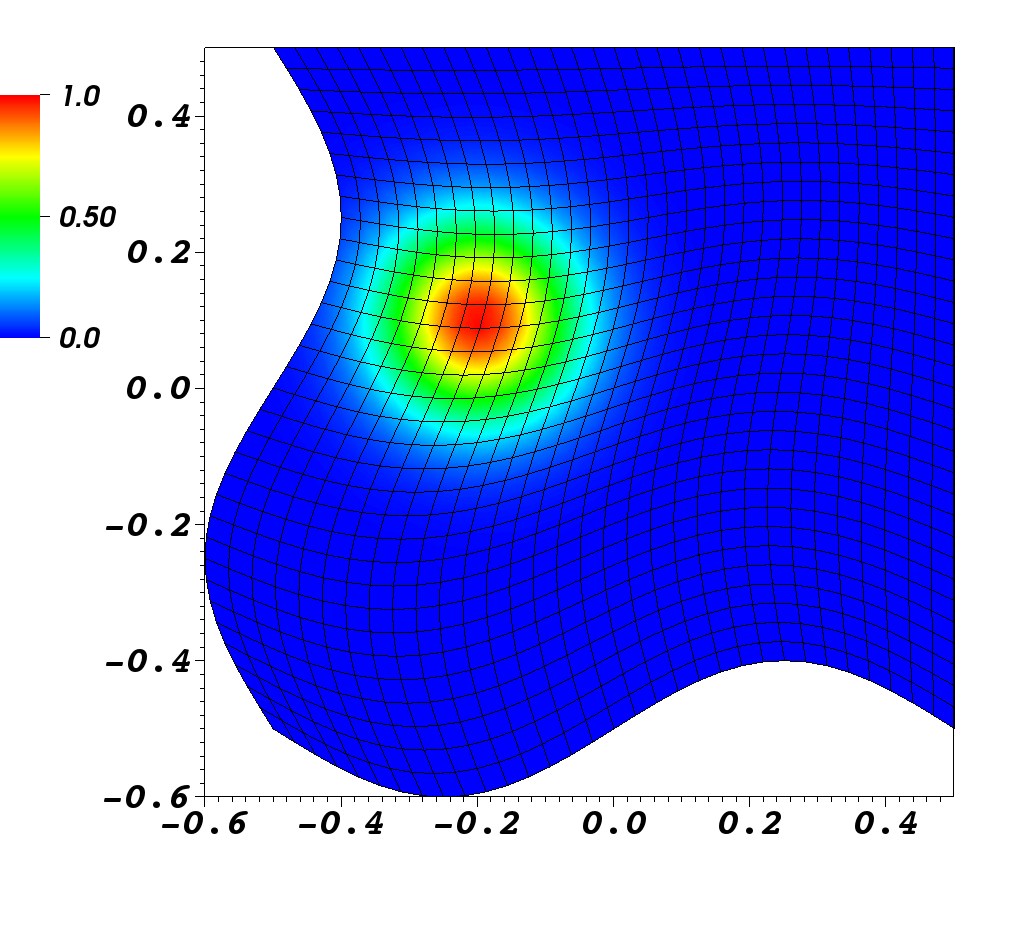}
    \includegraphics[width=.3\linewidth]{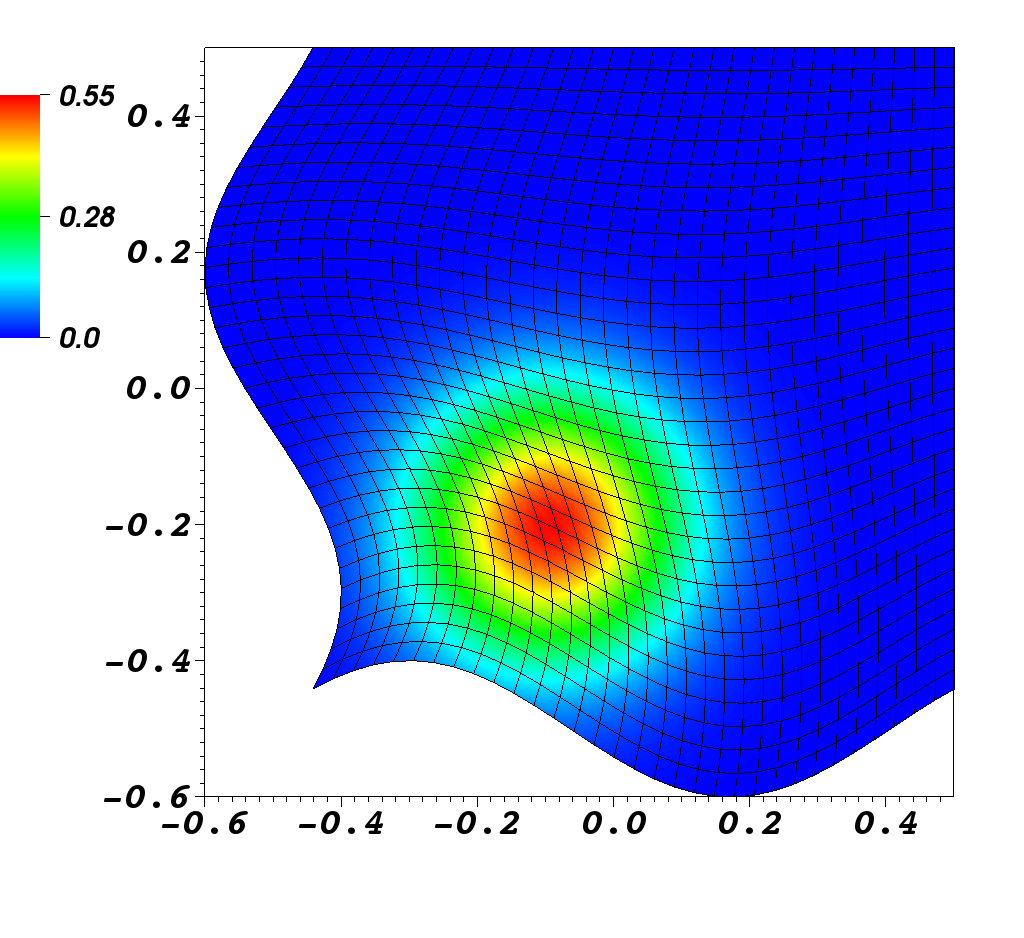}
    \includegraphics[width=.3\linewidth]{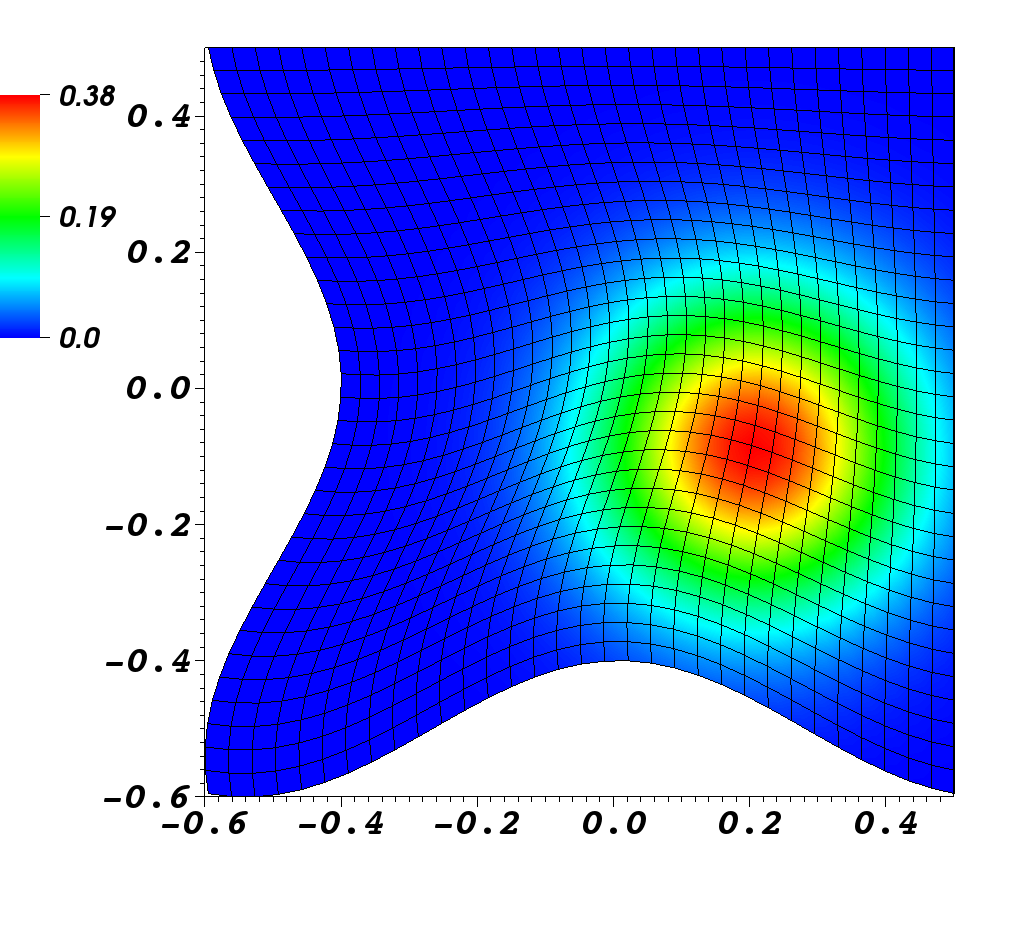}
    \caption{The mesh and solution for $\nu = 10^{-2}$ at time levels
      $t = 0, 0.4, 0.8$ (left to right).\label{fig:mesh_movement}}
  \end{center}
\end{figure}

\begin{center}
  \begin{table}[h]
    \centering
    \begin{tabular}{cc|cc|cc|cc}
      Cells per slab & Nr. of slabs   & $p = 1$  & rates  & $p = 2$  & rates & $p = 3$  & rates \\
      \hline
      $64$     & $8$       & 8.00e-2 &   -   & 1.52e-2 &   -  & 2.87e-3 &   -  \\
      $256$    & $16$      & 3.15e-2 & 1.3   & 3.24e-3 & 2.2  & 2.92e-4 & 3.3  \\
      $1024$   & $32$      & 1.30e-2 & 1.3   & 7.03e-4 & 2.2  & 3.21e-5 & 3.2  \\
      $4096$   & $64$      & 5.95e-3 & 1.1   & 1.64e-4 & 2.1  & 3.80e-6 & 3.1  \\
      \hline
      $64$     & $8$       & 1.75e-1 &   -   & 3.71e-2 &   -  & 6.67e-3 &   -  \\
      $256$    & $16$      & 7.78e-2 & 1.2   & 6.23e-3 & 2.6  & 5.60e-4 & 3.6  \\
      $1024$   & $32$      & 2.51e-2 & 1.6   & 1.03e-3 & 2.6  & 4.64e-5 & 3.6  \\
      $4096$   & $64$      & 7.60e-3 & 1.7   & 1.76e-4 & 2.5  & 3.88e-6 & 3.6
    \end{tabular}
    \caption{Rates of convergence when solving the advection--diffusion
      problem~\cref{eq:modprob}--~\cref{eq:boundaryinitialconditions} on a time-dependent
      domain with mesh deformation satisfying~\cref{eq:meshdeformation} with $\nu = 10^{-2}$
      (top) and $\nu = 10^{-6}$ (bottom). \label{tab:ex1_cfl1}}
  \end{table}
\end{center}

\section{Conclusions}
\label{s:conclusions}

In this paper, we presented and analyzed a space--time hybridizable
discontinuous Galerkin method for the advection--diffusion equation on
moving domains. We have shown the consistency, boundedness, and
stability of the bilinear form, and the well-posedness of the method
via an inf-sup condition. Further, we demonstrated the convergence of
the method and derived error estimates in a mesh dependent
norm. Theory was validated by a numerical example.

\bibliographystyle{siamplain}
\bibliography{references}
\end{document}